\documentclass[12pt]{amsart}
\usepackage{}
\usepackage{amsmath}
\usepackage{amsfonts}
\usepackage{amssymb}
\usepackage[all]{xy}           

\usepackage{bbding}
\usepackage{txfonts}
\usepackage{amscd}
\usepackage{xspace}

\usepackage[shortlabels]{enumitem}
\usepackage{ifpdf}
\ifpdf
\usepackage[colorlinks,final,backref=page,hyperindex]{hyperref}
\else
\usepackage[colorlinks,final,backref=page,hyperindex,hypertex]{hyperref}
\fi
\usepackage{tikz}
\usepackage[active]{srcltx}
\usepackage{amsmath,amssymb,mathrsfs,mathtools,microtype}
\makeatletter
\newtheorem{thm}{Theorem}[section]
\newtheorem{lem}[thm]{Lemma}
\newtheorem{cor}[thm]{Corollary}
\newtheorem{pro}[thm]{Proposition}
\theoremstyle{definition}
\newtheorem{defi}[thm]{Definition}
\newtheorem{ex}[thm]{Example}
\newtheorem{rmk}[thm]{Remark}
\usepackage{difftrees}

\newcommand {\emptycomment}[1]{}
\newcommand {\yh}[1]{{\marginpar{*}\scriptsize\textcolor{purple}{yh: #1}}}

\newcommand{\postLie}{\mathsf{PL}}

\newcommand{\nc}{\newcommand}
\nc{\delete}[1]{{}}
\nc{\CV}{\mathbf{C}}

\nc{\oprn}{\theta}
\nc{\Oprn}{\Theta}

\newcommand{\lon }{\,\rightarrow\,}
\newcommand{\be }{\begin{equation}}
	\newcommand{\ee }{\end{equation}}

\nc{\frakg}{\mathfrak{g}}
\newcommand{\dgla}{{\rm dgLa}}
\newcommand{\gla}{{\rm gLa}}
\newcommand{\sgla}{{\rm sgLa}}
\newcommand{\g}{\mathfrak g}
\newcommand{\h}{\mathfrak h}

\newcommand{\huaR}{\mathcal{R}}

\nc{\adj}{\xspace}
\nc{\opt}{operator\xspace}
\nc{\Opt}{Operator\xspace}

\newcommand{\huaU}{\mathcal{U}}

\newcommand{\huaX}{\mathcal{X}}

\nc{\calo}{\mathcal{O}}
\newcommand{\huaP}{\mathcal{P}}

\newcommand{\huaH}{\mathcal{H}}

\newcommand{\huaO}{{\mathcal{O}}}

\newcommand{\huaPostLie}{\mathsf{PostLie}}
\newcommand{\huaComTrias}{\mathsf{ComTrias}}
\newcommand{\huaCom}{\mathsf{Com}}

\newcommand{\huacom}{\mathsf{com}}

\newcommand{\Pe}{\mathsf{Perm}}
\newcommand{\huaVect}{\mathsf{Vect}}

\newcommand{\Sym}{\mathsf{S}}
\newcommand{\Ten}{\mathsf{T}}
\newcommand{\NR}{\mathsf{NR}}

\newcommand{\frki}{\mathfrak i}

\newcommand{\frkl}{\mathfrak l}

\newcommand{\frkC}{\mathfrak C}

\newcommand{\frkM}{\mathfrak M}

\newcommand{\frkR}{\mathfrak R}

\newcommand{\half}{\frac{1}{2}}

\newcommand{\Courant}[1]{\left\llbracket  #1\right\rrbracket }

\newcommand{\Id}{\rm{Id}}

\newcommand{\br}[1]{   [ \cdot,    \cdot  ]   }

\newcommand{\dM}{\mathrm{d}}

\newcommand{\Hom}{\mathrm{Hom}}

\newcommand{\Codiff}{\mathrm{Codiff}}
\newcommand{\dgOp}{\mathrm{dgOp}}
\newcommand{\dgCoop}{\mathrm{dgCoop}}
\newcommand{\Tw}{\mathrm{Tw}}
\newcommand{\B}{\mathrm{B}}

\newcommand{\PreLie}{\mathrm{PreLie}}
\newcommand{\PERM}{\mathrm{Perm}}

\newcommand{\coDer}{\mathsf{coDer}}
\newcommand{\Der}{\mathsf{Der}}

\newcommand{\Set}{\mathrm{Set}}

\newcommand{\Mod}{\mathrm{Mod}}
\newcommand{\End}{\mathrm{End}}
\newcommand{\ad}{\mathsf{ad}}
\newcommand{\pr}{\mathrm{pr}}

\newcommand{\Img}{\mathrm{Im}}

\newcommand{\K}{\bk}

\newcommand{\CE}{\mathsf{CE}}

\def\L{\mathcal{L}}

\newcommand{\bk}{{\mathbf{k}}}
\newcommand{\R}{\mathsf{R}}
\newcommand{\Def}{\mathsf{Def}}

\def\pcDGAloc{\operatorname{proArt_{loc}}}
\newcommand{\MCmodu}{\mathscr{MC}}

\begin{document}
	
	\title[Homotopy theory of post-Lie algebras]{Homotopy theory of post-Lie algebras}

	\author{Andrey Lazarev}
	\address{Department of Mathematics and Statistics, Lancaster University, Lancaster LA1 4YF, UK}
	\email{a.lazarev@lancaster.ac.uk}
	
	\author{Yunhe Sheng}
	\address{Department of Mathematics, Jilin University, Changchun 130012, Jilin, China}
	\email{shengyh@jlu.edu.cn}
	
	\author{Rong Tang}
	\address{Department of Mathematics, Jilin University, Changchun 130012, Jilin, China}
	\email{tangrong@jlu.edu.cn}
	
	
	
	\begin{abstract}
		In this paper, we study the homotopy theory of post-Lie algebras. Guided by Koszul duality theory, we consider the graded Lie algebra of coderivations of the cofree conilpotent graded cocommutative cotrialgebra generated by $V$. We show that in the case of $V$ being a shift of an ungraded vector space $W$, Maurer-Cartan elements of this graded Lie algebra are exactly  post-Lie algebra structures on $W$. The cohomology of a post-Lie algebra  is then defined using Maurer-Cartan twisting. The second cohomology group of a post-Lie algebra has a familiar interpretation as equivalence classes of infinitesimal deformations. Next we define a post-Lie$_\infty$ algebra structure  on a graded vector space to be a Maurer-Cartan element of the aforementioned graded Lie algebra. Post-Lie$_\infty$ algebras admit a useful characterization in terms of $L_\infty$-actions (or open-closed homotopy Lie algebras). Finally, we introduce the notion of homotopy Rota-Baxter operators on open-closed homotopy Lie algebras and show that certain homotopy Rota-Baxter operators induce post-Lie$_\infty$ algebras.	
	\end{abstract}

	
	\maketitle
	\tableofcontents

	\allowdisplaybreaks
	

\section{Introduction}

\subsection{Homotopy theories}

Homotopy invariant algebraic structures play a prominent role in modern mathematical physics  ~\cite{St19}.
Historically, the notion of $A_\infty$-algebras, which was introduced by Stasheff
in his study of based loop spaces~\cite{Sta63}, was the first such structure.
Relevant later  developments include the work of   Lada and Stasheff about
$L_\infty$-algebras in mathematical physics~\cite{LS}
and the work of Chapoton and Livernet about
pre-Lie$_\infty$ algebras~\cite{CL}. Strong homotopy (or infinity-) versions of a large class of algebraic
structures were studied in the context of
operads in ~\cite{LV,MSS,MV-1,MV-2}. A strong homotopy algebra is typically a Maurer-Cartan element in a certain differential graded (dg) Lie algebra (or possibly an $L_{\infty}$-algebra) and its Maurer-Cartan twisting is called the  cohomology of said strong homotopy algebra; it is known to control its deformation theory. Another important notion, particularly relevant to the present paper, is that of an $L_{\infty}$-action or, more generally, an open-closed homotopy Lie algebra introduced in  \cite{Lazarev,MZ,Mer-1}.

\subsection{Post-Lie algebras}

The notion of a post-Lie algebra has been introduced by Vallette in the course of study of Koszul duality of operads \cite{Val}.
Munthe-Kaas and his coauthors found that post-Lie algebras also naturally appear in differential geometry and numerical integration on  manifolds \cite{Munthe-Kaas-Lundervold}.  Meanwhile, it was found that post-Lie algebras play an essential role  in regularity structures in stochastic analysis \cite{BHZ,BK,Hairer}.
Recently, post-Lie algebras have been studied from different points of view including constructions of nonabelian generalized Lax pairs \cite{BGN},    Poincar\'e-Birkhoff-Witt type theorems \cite{D,Gubarev}, factorization theorems \cite{EMM} and relations to post-Lie groups \cite{AFM,BGST,MQS}. 

A pre-Lie algebra is a post-Lie algebra whose underlying Lie algebra is abelian; this notion is important in its own right~\cite{Bu,Ma,Sm22b}. The homotopy theory of pre-Lie algebras was given in \cite{CL} using the operadic approach. Furthermore, the cohomology theory of pre-Lie algebras was given by Dzhumadil'daev and Boyom in \cite{DA} and \cite{Boyom}.
It is, therefore, natural to develop the homotopy theory for post-Lie algebras, its cohomology theory and a strong homotopy version of a post-Lie algebra. The first step in this direction was \cite{TBGS}, where the notion of operator homotopy post-Lie algebras was introduced. A cohomology theory corresponding to operator homotopy post-Lie algebras was also given in \cite{TBGS}. In this paper,  the homotopy theory of post-Lie algebras is developed in full generality, including  the notion of post-Lie$_\infty$ algebra and the concomitant cohomology theory.

\subsection{Rota-Baxter operators}

The notion of Rota-Baxter operators on associative algebras was
introduced   by G. Baxter in his study of
fluctuation theory in probability \cite{Ba}. Rota-Baxter operators play important roles in Connes-Kreimer's~\cite{CK} algebraic
approach to the renormalization in perturbative quantum field
theory, the splitting of
operads~\cite{Aguiar,BBGN} and double Poisson algebras \cite{Goncharov,Goncharov-Gubarev}.
In the Lie algebra context, a Rota-Baxter operator was introduced as the
operator form of the classical Yang-Baxter equation.  To better understand the classical Yang-Baxter equation and
 related integrable systems, the more general notion of an $\huaO$-operator (also called relative Rota-Baxter operator)
on a Lie algebra was introduced by Kupershmidt~\cite{Ku}. Relative Rota-Baxter operators naturally give rise to pre-Lie algebras or post-Lie algebras \cite{Bai,BGN}. For further details on
Rota-Baxter operators, see ~\cite{Gub}.

The homotopy of Rota-Baxter associative algebras were studied in ~\cite{DK}, where it was noted that ``in general compact formulas
are yet to be found". Recently the homotopy theory of Rota-Baxter associative algebras was given in \cite{WZ}. For  Rota-Baxter  Lie algebras, one encounters a similarly challenging situation. Using higher derived brackets and  Maurer-Cartan elements approach,  the notion of a homotopy relative Rota-Baxter operator on an $L_\infty$-algebra with respect to a representation was formulated in \cite{LST}. Remarkably, strict homotopy relative Rota-Baxter operators give rise to pre-Lie$_\infty$ algebras. Later, homotopy relative Rota-Baxter operators on an $L_\infty$-algebra with respect to an $L_\infty$-action were studied in \cite{CC}. However, the underlying algebraic structures were still yet to be found. In this paper we introduce, on the one hand, the notion of homotopy Rota-Baxter operators on an open-closed homotopy Lie algebra using higher derived brackets; this contains the constructions given in \cite{CC} as special cases. On the other hand, we show that post-Lie$_\infty$ algebras are the underlying algebraic structures of certain homotopy Rota-Baxter operators.

\subsection{Main results and outline of the paper}
Since the Koszul dual of the operad $\huaPostLie$ for post-Lie algebras is the operad $\huaComTrias$ whose algebras are commutative trialgebras \cite{Val}, we are led to study graded Lie algebra of  coderivations of    the cofree conilpotent graded cocommutative cotrialgebra generated by a graded vector space $V$. If $V$ is a shift of an ordinary vector space $W$, Maurer-Cartan elements of this graded Lie algebra are exactly post-Lie algebra structures on $W$. Consequently, we obtain the dg Lie algebra that governs deformations of post-Lie algebras; the differential in it is the commutator with the given Maurer-Cartan element. This is the content of Section \ref{sec:con}.

In Section \ref{sec:coh}, we define the cohomology of a post-Lie algebra using the differential in the aforementioned dg Lie algebra that governs deformations of post-Lie algebras. We give a precise description of the coboundary operator, which can be viewed as a combination of the  Chevalley-Eilenberg coboundary operator for a Lie algebra and the coboundary operator for a pre-Lie algebra. As an application, we classify infinitesimal deformations of a post-Lie algebra in terms the second cohomology group of the constructed complex.

In Section \ref{sec:hom}, we define  post-Lie$_\infty$ algebra structures  as Maurer-Cartan elements of the aforementioned graded Lie algebra associated to a graded vector space. A post-Lie$_\infty$ algebra give rise to a sub-adjacent $L_\infty$-algebra which is itself part of a certain open-closed homotopy Lie algebra. We show how the original post-Lie$_\infty$ algebra is characterized in terms of this associated open-closed homotopy Lie algebra. Higher geometric structures, such as homotopy Poisson algebras and $L_\infty$-algebroids give rise to post-Lie$_\infty$ algebras naturally, providing a wealth of examples of these structures.

In Section \ref{sec:rb}, we introduce the notion of homotopy Rota-Baxter operators on an open-closed homotopy Lie algebra; these are certain Maurer-Cartan elements on $L_\infty$-algebras given by the higher derived brackets construction. Since actions of $L_\infty$-algebras are special open-closed homotopy Lie algebras, so homotopy Rota-Baxter operators on an open-closed homotopy Lie algebra  contain $\huaO$-operators on an $L_\infty$-algebra with respect to an action as special cases. In the classical case, a Rota-Baxter operator induces a post-Lie algebra. Now in the homotopy setting, we also show that certain homotopy Rota-Baxter operators induces post-Lie$_\infty$ algebras.	

	\subsection{Notation and conventions} We work in the category of
	cohomologically $\mathbb Z$-graded dg vector spaces $\operatorname{DGVect_{\bk}}$ (equivalently, cochain complexes) over a ground field
$\bk$ of characteristic zero.
Symbols $\otimes$ and $\Hom$ stand for tensor products and homomorphisms over $\bk$. A dg associative algebra is a dg vector space $A$ together with an associative and unital product $m:A\otimes A\to A$; formally an associative monoid in $\operatorname{DGVect_{\bk}}$. We will also need the notions of a commutative dg algebra (a commutative and associative monoid in $\operatorname{DGVect_{\bk}}$) and also that of a dg Lie algebra, which is a Lie object in $\operatorname{DGVect}$. Similarly, a (cocommutative) dg coalgebra is a comonoid in $\operatorname{DGVect_{\bk}}$. Thus, a dg coalgebra $C$ has a comultiplication $\Delta:C\to C\otimes C$. The $n$-fold iteration of $\Delta$ is the map $\Delta^{n}={(\Delta\otimes \Id\otimes\ldots\otimes \Id)\circ\ldots\circ\Delta}:C\to C^{\otimes (n+1)}$ and we say that $C$ is {\bf conilpotent} if for any $c\in C$ there exists $N\in \mathbb{Z}$ such that $\Delta^N(c)=0$.

Given a dg vector space $V$, we will denote by $\Ten(V)$ its tensor algebra and by $\Sym(V)$ its symmetric algebra which is the quotient of $\Ten(V)$ by the $2$-sided ideal of $\Ten(V)$ generated by all homogeneous elements $x,y\in V$ of the form
$
	x\otimes y-(-1)^{xy}y\otimes x.
$

We will write $\mathbb S_n$ for the group of permutations on $n$ symbols. A permutation $\sigma\in\mathbb S_n$ is called an $(i,n-i)$-shuffle if $\sigma(1)<\ldots <\sigma(i)$ and $\sigma(i+1)<\ldots <\sigma(n)$. If $i=0$ or $n$, we assume $\sigma=\Id$. The set of all $(i,n-i)$-shuffles will be denoted by $\mathbb S_{(i,n-i)}$. The notion of an $(i_1,\ldots,i_k)$-shuffle and the set $\mathbb S_{(i_1,\ldots,i_k)}$ are defined analogously.

Denote the product of homogeneous elements $v_1,\ldots,v_n\in V$ in $\Sym^n(V)$ by $v_1  \ldots  v_n$. The degree of $v_1 \ldots  v_n$ is by definition the sum of the degree of $v_i$. For a permutation $\sigma\in\mathbb S_n$ and $v_1,\ldots, v_n\in V$,  the Koszul sign $\varepsilon(\sigma;v_1,\ldots,v_n)\in\{-1,1\}$ is defined by
	\begin{eqnarray*}
		v_1\ldots v_n=\varepsilon(\sigma;v_1,\ldots,v_n)v_{\sigma(1)} \ldots  v_{\sigma(n)}.
	\end{eqnarray*}
	Moreover, there is a coalgebra structure on $\Sym(V)$  given by $\Delta(1_\bk)=1_\bk\otimes 1_\bk$ and
	\begin{eqnarray*}
		\Delta(v_1 \ldots  v_n)&=&\sum_{i=0}^{n}\sum_{\sigma\in\mathbb S_{(i,n-i)}}\varepsilon(\sigma;v_1,\ldots,v_n)\big(v_{\sigma(1)} \ldots  v_{\sigma(i)}\big)\otimes \big(v_{\sigma(i+1)} \ldots  v_{\sigma(n)}\big).
	\end{eqnarray*}
	The counit $\varepsilon:\Sym(V)\lon \bk$ is defined by
	$
		\varepsilon(1_\bk)=1_\bk$ and $\varepsilon(v_1 \ldots  v_n)=0.
$
	Next, the coaugmentation map  $\mu:\bk\lon \Sym(V)$ is defined by
	$
	\mu(1_\bk)=1_\bk.
	$
	Then $(\Sym(V),\Delta,\varepsilon,\mu)$ is the cofree conilpotent cocommutative coalgebra. More precisely, for any conilpotent cocommutative coalgebra $(C,\Delta_C,\varepsilon_C,\mu_C)$ and any linear map $f:C\lon V$ satisfying  $f(1_C)=0$, there exists a unique coalgebra homomorphism $\bar{f}:C\lon \Sym(V)$ such that $\pr_V\circ \bar{f}=f$,
	where $\bar{f}$ is given by
	\begin{eqnarray}
		\label{co-alg-homo}\bar{f}(1_C)=1_\bk,\,\,\bar{f}(x)=\sum_{i=1}^{+\infty}\sum_{x}\frac{1}{i!}f(x_{(1)}) \ldots f(x_{(i)}),\,\,\bar{\Delta}_C^{i-1}(x)=\sum_{x}x_{(1)}\otimes\ldots\otimes x_{(i)},
	\end{eqnarray}
	for all $x\in\bar{C}=\ker \varepsilon_C$. Here $\bar{\Delta}_C$ is the reduced coproduct $\bar{\Delta}_C:\bar{C}\lon \bar{C}\otimes \bar{C}$  given by
	\begin{eqnarray*}
		\bar{\Delta}_C(x)=\Delta_C(x)-x\otimes 1_C-1_C\otimes x,\,\,\,\,\forall x\in \bar{C}.
	\end{eqnarray*}

	\section{The controlling algebra of post-Lie algebras}\label{sec:con}

In this section, we associate to any graded vector space $V$ the graded Lie algebra of coderivations of the cofree conilpotent graded cocommutative cotrialgebra generated by $V$. We show that in the case of $V$ being a shift of an ungraded vector space $W$, this graded Lie algebra is exactly the controlling algebra of post-Lie algebras, i.e. its Maurer-Cartan elements are post-Lie algebra structures on $W$.
	
	\begin{defi} (\cite{Val})\label{post-lie-defi}
		A {\bf post-Lie algebra} $(\g,[\cdot,\cdot]_\g,\rhd)$ consists of a Lie algebra $(\g,[\cdot,\cdot]_\g)$ and a binary product $\rhd:\g\otimes\g\lon\g$ such that
		\begin{eqnarray}
			\label{Post-1}x\rhd[y,z]_\g&=&[x\rhd y,z]_\g+[y,x\rhd z]_\g,\\
			\label{Post-2}([x,y]_\g+x\rhd y-y\rhd x) \rhd z&=&x\rhd(y\rhd z)-y\rhd(x\rhd z),\,\,\forall x, y, z\in \g.
		\end{eqnarray}
	\end{defi}

		Any post-Lie algebra $(\g,[\cdot,\cdot]_\g,\rhd)$ gives rise to a new Lie algebra
		$\g_\rhd:=(\g,[\cdot,\cdot]_{\g_\rhd})$
		defined by
		$$[x,y]_{\g_\rhd} \coloneqq x\rhd y-y\rhd x+[x,y]_\g,\quad\forall x,y\in\g,$$
		which is called the {\bf sub-adjacent Lie algebra}. Moreover,
		Eqs.~\eqref{Post-1}-\eqref{Post-2} equivalently mean that the linear map $L:\g\to\mathfrak g\mathfrak l(\g)$ defined by $L_{x}y=x\rhd y$ is an action of $(\g,[\cdot,\cdot]_{\g_\rhd})$ on $(\g,[\cdot,\cdot]_\g)$.
	
	\begin{rmk}
		Let $(\g,[\cdot,\cdot]_\g,\rhd)$ be a post-Lie algebra. If the Lie bracket $[\cdot,\cdot]_\g$ vanishes, then $(\g,\rhd)$ becomes a pre-Lie algebra. Thus,  a post-Lie algebra can be viewed as a nonabelian generalization of a pre-Lie algebra.
	\end{rmk}

	\begin{defi}\label{post-lie-homo-defi}
		\begin{itemize}
			\item  A {\bf derivation} of a post-Lie algebra $(\g,[\cdot,\cdot]_\g,\rhd)$ is a linear map $\varphi:\g\lon\g$  satisfying $
				\varphi([x,y]_\g)=[\varphi(x),y]_\g+[x,\varphi(y)]_\g$ and $
				\varphi(x\rhd y)=\varphi(x)\rhd y+x\rhd \varphi(y)$ for all $x,y\in\g.$
			Denote by $\Der(\g)$ the set of derivations of the post-Lie algebra $(\g,[\cdot,\cdot]_\g,\rhd)$.
			
			\item   A {\bf homomorphism} $\phi$ from a post-Lie algebra $(\g,[\cdot,\cdot]_\g,\rhd_\g)$ to a post-Lie algebra $(\h,[\cdot,\cdot]_\h,\rhd_\h)$ is a Lie algebra morphism $\phi:\g\lon \h$ such that
			$
				\phi(x\rhd_\g y)=\phi(x)\rhd_\h \phi(y)$ for all $x,y\in\g.
		$
			In particular, if $\phi$ is  invertible,  then $\phi$ is called an  {\bf isomorphism}.
		\end{itemize}
		
	\end{defi}
	
	\begin{ex}
		{\rm
			Let $(V,\rhd)$ be a  magmatic algebra, i.e. $V$ is a vector space and $\rhd:V \otimes V\to V$ is an arbitrary linear operation, \cite[Definition 13.8.1]{LV}. Extend the  magmatic operation $\rhd:V \otimes V\to V$
			to the free Lie algebra $(Lie(V),[\cdot,\cdot])$. Then $(Lie(V),[\cdot,\cdot],\rhd)$ is a post-Lie algebra.			
		}
	\end{ex}

	\begin{ex}[Free post-Lie algebra]\label{free-post-Lie}
		Let $\huaO$ be the set of isomorphism classes of planar rooted trees:
		\[
		\huaO= \Big\{\begin{array}{c}
			\scalebox{0.6}{\ab}, \scalebox{0.6}{\aabb},
			\scalebox{0.6}{\aababb}, \scalebox{0.6}{\aaabbb},\scalebox{0.6}{\aabababb},
			\scalebox{0.6}{\aaabbabb},
			\scalebox{0.6}{\aabaabbb}, \scalebox{0.6}{\aaababbb}, \scalebox{0.6}{\aaaabbbb},\scalebox{0.6}{\aababababb},\scalebox{0.6}{\aaabbababb},\scalebox{0.6}{\aabaabbabb},\scalebox{0.6}{\aababaabbb},\scalebox{0.6}{\aaababbabb},\scalebox{0.6}{\aabaababbb},\scalebox{0.6}{\aaabbaabbb},\scalebox{0.6}{\aaabababbb},\scalebox{0.6}{\aaaabbbabb},\scalebox{0.6}{\aabaaabbbb},\scalebox{0.6}{\aaaababbbb},\scalebox{0.6}{\aaaabbabbb},\scalebox{0.6}{\aaabaabbbb},\scalebox{0.6}{\aaaaabbbbb},\ldots
		\end{array}
		\Big\}.
		\]
		Let $\bk\{\huaO\}$ be the free $\bk$-vector space generated by  $\huaO$. The {\bf left grafting operator} $\rhd:\bk\{\huaO\}\otimes \bk\{\huaO\}\to \bk\{\huaO\}$ is defined by		
		\begin{eqnarray}\label{free-post-product}
			\tau\rhd \omega=\sum_{s\in {\rm Nodes}(\omega)}\tau\circ_{s}\omega,\,\,\forall \tau,\omega\in \huaO,
			\vspace{-.1cm}
		\end{eqnarray}
		where $\tau\circ_{s}\omega$ is the planar  rooted tree resulting from attaching the root of $\tau$ to  the node $s$ of the tree $\omega$ from the left. Consider the free Lie algebra $Lie(\bk\{\huaO\})$ generated by the vector space $\bk\{\huaO\}$ and extend the left grafting operator $\rhd$ on $\bk\{\huaO\}$ to the free Lie algebra $Lie(\bk\{\huaO\})$ by \eqref{Post-1} and \eqref{Post-2}. Then $(Lie(\bk\{\huaO\}),[\cdot,\cdot],\rhd)$ is a post-Lie algebra, which is the free post-Lie algebra generated by  one generator $\{\scalebox{0.6}{\ab}\}$ \cite{Munthe-Kaas-Lundervold,Val}.
	\end{ex}

	\begin{ex}[Ihara Lie algebra]
		{\rm
			Any   element $f$ of the free  Lie algebra $Lie(\bk\{x,y\})$  gives rise to a derivation $D_f$  by
			$$
			D_f(x)=0,\,\,\,\,D_f(y)=[y,f].
			$$
			Moreover, for all $f,g\in Lie(\bk\{x,y\})$,  define  $\rhd:Lie(\bk\{x,y\})\otimes Lie(\bk\{x,y\})\lon Lie(\bk\{x,y\})$ by
			\begin{eqnarray*}
				f\rhd g=D_f(g),\,\,\,\,\forall f,g\in Lie(\bk\{x,y\}).
			\end{eqnarray*}
			Then $(Lie(\bk\{x,y\}),[\cdot,\cdot],\rhd)$ is a post-Lie algebra, whose sub-adjacent Lie algebra is exactly the  Ihara Lie algebra \cite{Schneps}. 
		}
	\end{ex}
	
	\emptycomment{
		\begin{pro}
			Let $(\g,[\cdot,\cdot]_\g)$ be a Lie algebra, $(A,\cdot)$ a commutative associative algebra and $\rho:\g\lon \Der(A)$ a Lie algebra homomorphism. Then  there is a post-Lie algebra structure on $A\otimes \g$  given by
			\begin{eqnarray}
				\label{p-1}[a\otimes x,b\otimes y]&=&a\cdot b\otimes [x,y]_\g,\\
				\label{p-2}(a\otimes x)\rhd (b\otimes y)&=&a\cdot \rho(x)(b)\otimes y,\,\,\,\,\forall a,b\in A,~x,y\in\g.
			\end{eqnarray}
		\end{pro}
		
		\begin{cor}
			Let $(P,[\cdot,\cdot],\cdot)$ be a Poisson algebra. Then  there is a post-Lie algebra structure on $P\otimes P$  given by
			\begin{eqnarray}
				\label{p-1}[x\otimes y,z\otimes w]&=&x\cdot z\otimes [y,w],\\
				\label{p-2}(x\otimes y)\rhd (z\otimes w)&=&x\cdot [y,z]\otimes w,\,\,\,\,\forall x,y,z,w\in P.
			\end{eqnarray}
		\end{cor}
		
		\begin{cor}
			Let $({A},\L,[\cdot,\cdot]_\L,\alpha)$ be a Lie-Rinehart algebra. Then  there is a post-Lie algebra structure on $A\otimes \L$  given by
			\begin{eqnarray}
				\label{p-1}[a\otimes x,b\otimes y]&=&ab\otimes [x,y]_\L,\\
				\label{p-2}(a\otimes x)\rhd (b\otimes y)&=&a\alpha(x)(b)\otimes y,\,\,\,\,\forall a,b\in A,~x,y\in\L.
			\end{eqnarray}
		\end{cor}
	}

	Post-Lie algebras are algebras over the operad $\huaPostLie$. Its Koszul dual is the operad $\huaComTrias$ whose algebras are commutative trialgebras, \cite{Val}. We will now recall the definition of a commutative trialgebra.
	
	\begin{defi}
		A {\bf commutative trialgebra} is a triple $(A,\cdot,\ast)$, where $(A,\cdot)$ is a commutative associative algebra  and
		$(A,\ast)$ is a permutative  algebra, 
		such that the following compatibility relations hold:
		\begin{eqnarray}
			(x\cdot y)\ast z&=&(x\ast y)\ast z\\
			x\ast (y\cdot z)&=&(x\ast y)\cdot z.
		\end{eqnarray}
	\end{defi}
	
	\begin{ex}{\rm(\cite[Theorem 28]{Val})}\label{free-comtrial} {\rm
			Let $V$ be a graded vector space. There is a commutative trialgebra structure on $\Sym(V)\otimes \bar{\Sym}(V)$  given by
			\begin{eqnarray}
				\label{ComTrias-1}\big(x_1\ldots x_m\otimes y_1\ldots y_n\big)\cdot \big(u_1\ldots u_k\otimes v_1\ldots v_l\big)&=&(-1)^{\alpha}x_1\ldots x_m u_1\ldots u_k\otimes y_1\ldots y_n v_1\ldots v_l,\\
				\label{ComTrias-2}\quad\big(x_1\ldots x_m\otimes y_1\ldots y_n\big)\ast \big(u_1\ldots u_k\otimes v_1\ldots v_l\big)&=&x_1\ldots x_m y_1\ldots y_nu_1\ldots u_k \otimes v_1\ldots v_l,
			\end{eqnarray}
			where $\alpha=(y_1+\ldots+y_n)(u_1+\ldots +u_k)$.
			Moreover, $(\Sym(V)\otimes \bar{\Sym}(V),\cdot,\ast)$ is the free graded  commutative trialgebra generated by the graded vector space $V$ and we denote it by $\huaComTrias(V)$.
		}
	\end{ex}
	
	\begin{defi}
		A {\bf derivation} of a graded commutative trialgebra $(A,\cdot,\ast)$ is a graded linear map $d:A\lon A$ such that for all $x,y\in A,$ the following equalities hold:
		\begin{eqnarray}
			\label{der-1}d(x\cdot y)&=&d(x)\cdot y+(-1)^{xd}x\cdot d(y),\\
			\label{der-2}d(x\ast y)&=&d(x)\ast y+(-1)^{xd}x\ast d(y).
		\end{eqnarray}
			\end{defi}
Let $d\in \Hom(V,\Sym(V)\otimes \bar{\Sym}(V))$. Define  $\Phi(d):\Sym(V)\otimes \bar{\Sym}(V)\lon \Sym(V)\otimes \bar{\Sym}(V)$ by
	\begin{eqnarray}\label{der-free-comtrial}
		\nonumber&&\Phi(d)(x_1\ldots x_m\otimes y_1\ldots y_n)\\
		&=&\sum_{i=1}^{m}(-1)^{(x_1+\ldots+x_{i-1})d}\big(1_\bk\otimes x_1\ast\ldots\ast d(x_i)\ast\ldots\ast 1_\bk\otimes x_m\big)\ast\big(1_\bk\otimes y_1\cdot\ldots\cdot 1_\bk\otimes y_n\big)\\
		\nonumber&+&\sum_{j=1}^{n}(-1)^{(x_1+\ldots+x_m+y_1+\ldots+y_{j-1})d}\big(1_\bk\otimes x_1\ast\ldots\ast 1_\bk\otimes x_m\big)\ast\big(1_\bk\otimes y_1\cdot\ldots \cdot d(y_j)\cdot\ldots\cdot 1_\bk\otimes y_n\big)
	\end{eqnarray}
	for all $x_1,\ldots,x_m,y_1,\ldots, y_n\in V.$
	
	\begin{pro}
		With notations as above, $\Phi$ is an isomorphism from the graded vector space $\Hom(V,\Sym(V)\otimes \bar{\Sym}(V))$ to $\Der(\Sym(V)\otimes \bar{\Sym}(V))$.
	\end{pro}
	
	\begin{proof}
		We define a graded linear map $\frki:\Der(\Sym(V)\otimes \bar{\Sym}(V))\lon \Hom(V,\Sym(V)\otimes \bar{\Sym}(V))$ as follows:
		\begin{eqnarray}\label{der-free-comtrial-re}
			\frki(f)=f\circ i_V,\,\,\,\,\forall f\in\Der(\Sym(V)\otimes \bar{\Sym}(V)).
		\end{eqnarray}
		Here $i_V:V\lon \Sym(V)\otimes \bar{\Sym}(V))$ is the embedding given by
	$
			i_V(v)=1_\bk\otimes v$ for all $v\in V.
	$
		For any $d\in \Hom(V,\Sym(V)\otimes \bar{\Sym}(V))$, by \eqref{der-free-comtrial} and \eqref{der-free-comtrial-re}, we have
		\begin{eqnarray*}
			\big((\frki\circ\Phi)(d)\big)(v)=\Phi(d)(i_V(v))=\Phi(d)(1_\bk\otimes v)=d(v).
		\end{eqnarray*}
		Therefore $\frki\circ\Phi=\Id$. On the other hand, for any $f\in\Der(\Sym(V)\otimes \bar{\Sym}(V))$, we obtain that
	\begin{eqnarray*}
			&&f(x_1\ldots x_m\otimes y_1\ldots y_n)\\
			&=&f\big((1_\bk\otimes x_1\ast\ldots\ast 1_\bk\otimes x_m)\ast(1_\bk\otimes y_1\cdot\ldots\cdot 1_\bk\otimes y_n)\big)\\
			&=&f\big((1_\bk\otimes x_1\ast\ldots\ast 1_\bk\otimes x_m)\big)\ast(1_\bk\otimes y_1\cdot\ldots\cdot 1_\bk\otimes y_n)\\
			&&+(-1)^{(x_1+\ldots+x_m)d}(1_\bk\otimes x_1\ast\ldots\ast 1_\bk\otimes x_m)\ast f\big((1_\bk\otimes y_1\cdot\ldots\cdot 1_\bk\otimes y_n)\big)\\
			&{=}&\sum_{i=1}^{m}(-1)^{(x_1+\ldots+x_{i-1})d}\big(1_\bk\otimes x_1\ast\ldots\ast f(1_\bk\otimes x_i)\ast\ldots\ast 1_\bk\otimes x_m\big)\ast\big(1_\bk\otimes y_1\cdot\ldots\cdot 1_\bk\otimes y_n\big)\\
			\nonumber&+&\sum_{j=1}^{n}(-1)^{(x_1+\ldots+x_m+y_1+\ldots+y_{j-1})d}\big(1_\bk\otimes x_1\ast\ldots\ast 1_\bk\otimes x_m\big)\ast\big(1_\bk\otimes y_1\cdot\ldots \cdot f(1_\bk\otimes y_j)\cdot\ldots\cdot 1_\bk\otimes y_n\big).
		\end{eqnarray*}
		Moreover,   \eqref{der-free-comtrial} implies that $f=\Phi(\frki(f))$, i.e. $\Phi\circ\frki=\Id$. This finishes the proof.
	\end{proof}

	\begin{defi}
		A {\bf cocommutative cotrialgebra} is a triple $(C,\Delta,\delta)$, where $(C,\Delta)$ is a cocommutative coassociative coalgebra:
		\begin{eqnarray}
			\Delta=\tau\circ\Delta,\,\,\,\,(\Id\otimes\Delta)\circ\Delta=(\Delta\otimes\Id)\circ\Delta,
		\end{eqnarray}
		$(C,\delta)$ is a copermutative  coalgebra:
		\begin{eqnarray}
			(\Id\otimes\delta)\circ\delta=(\delta\otimes\Id)\circ\delta=(\tau\otimes\Id)\circ(\delta\otimes\Id)\circ\delta,
		\end{eqnarray}
		and the following compatibility relations hold:
		\begin{eqnarray}
			(\Delta\otimes\Id)\circ\delta&=&(\delta\otimes\Id)\circ\delta,\\
			(\delta\otimes\Id)\circ\Delta&=&(\Id\otimes\Delta)\circ\delta,
		\end{eqnarray}
		where $\tau(a\otimes b)=(-1)^{ab}b\otimes a$.
	\end{defi}

	A cocommutative cotrialgebra $(C,\Delta,\delta)$ is called conilpotent if for any $x\in C$ there exists $N\in\mathbb{Z}$ such that $\theta^{n}(x)=0$ for $n>N$, where $\theta^{n}:C\to C^{\otimes (n+1)}$ is any $n$-fold iteration of  $\Delta$ or $\delta$.

	\begin{thm}
		Let $V$ be a graded vector space. Then there is a cocommutative cotrialgebra structure on $\Sym(V)\otimes \bar{\Sym}(V)$   given by
		\begin{eqnarray*}
			\label{coComTrias-1}\Delta(x_1\ldots x_m\otimes y_1\ldots y_n)&=&\sum_{0\le i\le m\atop 1\le j\le n-1}\sum_{\sigma\in\mathbb S_{(i,m-i)}\atop\tau\in\mathbb S_{(j,n-j)}}\varepsilon(\sigma)\varepsilon(\tau)(-1)^{(x_{\sigma(i+1)}+\ldots+x_{\sigma(m)})(y_{\tau(1)}+\ldots+y_{\tau(j)})}\\
			&&\big(x_{\sigma(1)}\ldots x_{\sigma(i)}\otimes y_{\tau(1)}\ldots y_{\tau(j)}\big)\otimes  \big(x_{\sigma(i+1)}\ldots x_{\sigma(m)}\otimes y_{\tau(j+1)}\ldots y_{\tau(n)}\big),\\
			\label{coComTrias-2}\delta(x_1\ldots x_m\otimes y_1\ldots y_n)&=&\sum_{0\le i\le m-1\atop 1\le j\le m}\sum_{\sigma\in\mathbb S_{(i,j,m-i-j)}}\varepsilon(\sigma)\\
			&&\big(x_{\sigma(1)}\ldots x_{\sigma(i)}\otimes x_{\sigma(i+1)}\ldots x_{\sigma(i+j)}\big)\otimes  \big(x_{\sigma(i+j+1)}\ldots x_{\sigma(m)}\otimes y_1\ldots y_n\big).
		\end{eqnarray*}
		Moreover, $(\Sym(V)\otimes \bar{\Sym}(V),\Delta,\delta)$ is the cofree conilpotent graded cocommutative cotrialgebra generated by the graded vector space $V$ and we denote it by $\huaComTrias^c(V)$.
	\end{thm}
	\begin{proof}
	Since $(\Sym(V)\otimes \bar{\Sym}(V),\Delta,\delta)$ is the graded dual of the free graded  commutative trialgebra $(\Sym(V^*)\otimes \bar{\Sym}(V^*),\cdot,\ast)$, it follows that $(\Sym(V)\otimes \bar{\Sym}(V),\Delta,\delta)$ is a graded  cocommutative cotrialgebra. It is clear that it is conilpotent. Let us show that it is cofree. For $m,n\geq1$, we write $$\Delta^{n}=(\Delta\otimes{\Id}^{\otimes n-1})\ldots(\Delta\otimes{\Id})\Delta,\quad\delta^{m}=(\delta\otimes{{\Id}}^{\otimes m-1})\ldots(\delta\otimes\Id)\delta.$$
	\emptycomment{	We have
		\begin{eqnarray*}
			\Delta^{n}(x_1\ldots x_m\otimes y_1\ldots y_n)&=&0,\\
			\delta^{m}(x_1\ldots x_m\otimes y_1\ldots y_n)&=&0.
		\end{eqnarray*}
		\textcolor{red}{By the definition of $\Delta$ and $\delta$, we deduce that  $\theta^{m+n}(x_1\ldots x_m\otimes y_1\ldots y_n)=0$ which implies that $(\Sym(V)\otimes \bar{\Sym}(V),\Delta,\delta)$ is a graded  conilpotent cocommutative cotrialgebra.
		}}
			Then we have
		\begin{eqnarray}\label{product-coproduct}
			\ast(\ast_{m}\otimes \cdot_{n})((\delta^{m-1}\otimes \Delta^{n-1})\circ\delta)(x_1\ldots x_m\otimes y_1\ldots y_n)=m!n!x_1\ldots x_m\otimes y_1\ldots y_n,
		\end{eqnarray}
		where
		\begin{eqnarray*}
			\ast_{m}(X_1\otimes Y_1,\ldots,X_m\otimes Y_m)&=&X_1Y_1\ldots Y_{m-1}X_m\otimes Y_m,\\
			\cdot_{n}(X_1\otimes Y_1,\ldots,X_n\otimes Y_n)&=&(-1)^{\sum_{i=1}^{n}(Y_1+\ldots+Y_{i-1})X_i}X_1\ldots X_n\otimes Y_1\ldots Y_n.
		\end{eqnarray*}
		Let $(C,\Delta,\delta)$ be a graded  conilpotent cocommutative cotrialgebra and $f:C\lon \Sym(V)\otimes \bar{\Sym}(V)$ be a coalgebra homomorphism. Note that $f\in\Hom^0(C,\Sym(V)\otimes \bar{\Sym}(V))$ is the sum of $f^{m,n}:C\lon \Sym^m(V)\otimes \Sym^n(V),~\forall m\ge0,n\ge1$. Moreover, denote the projection $\Sym(V)\otimes \bar{\Sym}(V)\lon \Sym^m(V)\otimes \Sym^n(V)$ by $\pr^{m,n}_V$. Then $f^{0,1}=\pr^{0,1}_V\circ f$. Since $f$ is a coalgebra homomorphism, we have the following diagram:
		\begin{equation}\label{cotri-homor}
			\begin{split}
				\xymatrix{ C \ar^{f}[rr]\ar_{\text{$(\delta^{m-1}\otimes \Delta^{n-1})\circ\delta$}}[d] && \Sym(V)\otimes \bar{\Sym}(V) \ar^{\text{$(\delta^{m-1}\otimes \Delta^{n-1})\circ\delta$}}[d] \\
					\text{$C^{\otimes m+n}$} \ar^{f^{\otimes m+n}}[rr]&& \text{$(\Sym(V)\otimes \bar{\Sym}(V))^{\otimes m+n}$}
					.}
			\end{split}
		\end{equation}
		By \eqref{product-coproduct} and \eqref{cotri-homor}, we obtain that $$\pr^{m,n}_V\circ \big(\ast(\ast_{m}\otimes \cdot_{n})\big)\circ f^{\otimes m+n}\circ (\delta^{m-1}\otimes \Delta^{n-1})\circ\delta=m!n!f^{m,n}.$$ Therefore, we deduce that
		\begin{eqnarray*}
			f^{m,n}(x)&=&\sum_{x}\frac{1}{m!n!}\ast(\ast_{m}\otimes \cdot_{n})\Big(\big(1_\bk\otimes f^{0,1}(x_{(1)})\big)\otimes \big(1_\bk\otimes f^{0,1}(x_{(2)})\big)\otimes\ldots\otimes\big(1_\bk\otimes f^{0,1}(x_{(m+n)})\big)\Big)\\
			&=&\sum_{x}\frac{1}{m!n!}f^{0,1}(x_{(1)})\ldots f^{0,1}(x_{(m)})\otimes f^{0,1}(x_{(m+1)})\ldots f^{0,1}(x_{(m+n)}),
		\end{eqnarray*}
		where $(\delta^{m-1}\otimes \Delta^{n-1})\delta(x)=\sum_{x}x_{(1)}\otimes x_{(2)}\otimes\ldots\otimes x_{(m+n)},$ for all $x\in C.$ Moreover, we can reconstruct the coalgebra homomorphism $f$ by $f^{0,1}$ as follows:
		\begin{eqnarray}\label{cofree-cotri-homo}
			f(x)=\sum_{m\ge 0,n\ge 1}\sum_{x}\frac{1}{m!n!}f^{0,1}(x_{(1)})\ldots f^{0,1}(x_{(m)})\otimes f^{0,1}(x_{(m+1)})\ldots f^{0,1}(x_{(m+n)}).
		\end{eqnarray}
		Therefore, for any graded  conilpotent cocommutative cotrialgebra $(C,\Delta,\delta)$ and any linear map $\phi:C\lon V$, by \eqref{cofree-cotri-homo}, there exists a unique coalgebra homomorphism $\Phi:C\lon \Sym(V)\otimes \bar{\Sym}(V)$ such that we have the following commutative diagram:
		$$
		\xymatrix{
			C\ar[rr]^{\Phi}\ar[dr]_{\phi} & & \Sym(V)\otimes \bar{\Sym}(V)\ar[dl]^{\pr^{0,1}_V}. \\
			& V&
		}
		$$
		Here $\Phi$ is given by
		\begin{eqnarray}
			\Phi(x)=\sum_{m\ge 0,n\ge 1}\sum_{x}\frac{1}{m!n!}\phi(x_{(1)})\ldots \phi(x_{(m)})\otimes \phi(x_{(m+1)})\ldots \phi(x_{(m+n)}).
		\end{eqnarray}
		This shows that   $(\Sym(V)\otimes \bar{\Sym}(V),\Delta,\delta)$ is the cofree  conilpotent cocommutative cotrialgebra generated by the graded vector space $V$.
	\end{proof}

	\begin{defi}
		A {\bf coderivation} of a graded cocommutative cotrialgebra $(C,\Delta,\delta)$ is a graded linear map $D:C\lon C$ such that
		\begin{eqnarray}
			\label{coder-1}\Delta\circ D&=&(D\otimes {\Id}+{\Id}\otimes D)\circ \Delta,\\
			\label{coder-2}\delta\circ D&=&(D\otimes {\Id}+{\Id}\otimes D)\circ \delta.
		\end{eqnarray}
	\end{defi}
	
	Note that an element $\frkR\in\Hom^k(\Sym(V)\otimes \bar{\Sym}(V),V)$ is the sum of $\frkR_{p,q}:\Sym^p(V)\otimes \Sym^q(V)\lon V$. Write  $\frkR=\sum_{p\ge 0,q\ge 1} \frkR_{p,q}$ and  $$\Hom(\Sym(V)\otimes \bar{\Sym}(V),V)=\oplus_{n\in\mathbb Z}\Hom^n(\Sym(V)\otimes \bar{\Sym}(V),V),$$ the decomposition of the space of homomorphisms into graded pieces.
	For $\frkR_{p,q}:\Sym^p(V)\otimes \Sym^q(V)\lon V$, we define  $\hat{\frkR}_{p,q}\in \Hom(\Sym(V)\otimes \bar{\Sym}(V),\Sym(V)\otimes \bar{\Sym}(V))$ by
	\begin{eqnarray}
		\nonumber&&\hat{\frkR}_{p,q}(x_1\ldots x_n\otimes x_{n+1}\ldots x_{n+m})\\
		\label{coder-post-lie}&=&\sum_{\sigma\in \mathbb S_{(p,q,n-p-q)}}\varepsilon(\sigma)\frkR_{p,q}(x_{\sigma(1)} \ldots x_{\sigma(p)}\otimes x_{\sigma(p+1)} \ldots x_{\sigma(p+q)})x_{\sigma(p+q+1)}\ldots x_{\sigma(n)}\otimes x_{n+1}\ldots x_{n+m}\\
		\nonumber&&+\sum_{\sigma\in\mathbb S_{(n-p,p)}}\sum_{ \tau\in\mathbb S_{(q,m-q)}}\varepsilon(\sigma)\varepsilon(\tau)(-1)^{k(x_{\sigma(1)}+\ldots+x_{\sigma(n-p)})}\\
		\nonumber&&x_{\sigma(1)}\ldots x_{\sigma(n-p)}\otimes \frkR_{p,q}\big(x_{\sigma(n-p+1)} \ldots x_{\sigma(n)}\otimes x_{n+\tau(1)} \ldots x_{n+\tau(q)}\big)x_{n+\tau(q+1)}\ldots x_{n+\tau(m)},
	\end{eqnarray}
	for all $n\ge p,m\ge q,$ and
	$\hat{\frkR}_{p,q}(x_1\ldots x_n\otimes x_{n+1}\ldots x_{m+n})=0$ for other cases.
	
	The following result holds.
	
	\begin{pro}\label{important-iso}
		  Define $\Psi:\Hom(\Sym(V)\otimes \bar{\Sym}(V),V)\lon \coDer(\huaComTrias^c(V))	$	  by
		\begin{eqnarray}
			\nonumber&&\Psi(\frkR)(x_1\ldots x_n\otimes x_{n+1}\ldots x_{n+m})=\sum_{p\ge 0,q\ge 1}\hat{\frkR}_{p,q}(x_1\ldots x_n\otimes x_{n+1}\ldots x_{n+m}).
		\end{eqnarray}
Then $\Psi$ is an isomorphism of graded vector spaces.
	\end{pro}
	\begin{proof}
		Let $D\in \coDer(\huaComTrias^c(V))$. Then $D$ is the sum of $D^{m,n}:\huaComTrias^c(V)\lon \Sym^m(V)\otimes \Sym^n(V),$  for all $m\ge0,n\ge1$. 
		Since $D$ is a coderivation, we have the following commutative diagram:
		\begin{equation}\label{cotri-coder}
			\begin{split}
				\xymatrix{ \huaComTrias^c(V) \ar^{D}[rr]\ar_{\text{$(\delta^{m-1}\otimes \Delta^{n-1})\circ\delta$}}[d] && \huaComTrias^c(V) \ar^{\text{$(\delta^{m-1}\otimes \Delta^{n-1})\circ\delta$}}[d] \\
					\text{$(\huaComTrias^c(V))^{\otimes m+n}$} \ar_{\sum_{i=1}^{m+n}{\Id}^{\otimes i-1}\otimes D\otimes {\Id}^{\otimes m+n-i}}[rr]&& \text{$(\huaComTrias^c(V))^{\otimes m+n}$}
					.}
			\end{split}
		\end{equation}
		By \eqref{product-coproduct} and \eqref{cotri-coder}, we obtain that $$\pr^{m,n}_V\circ \big(\ast(\ast_{m}\otimes \cdot_{n})\big)\circ \Big(\sum_{i=1}^{m+n}{\Id}^{\otimes i-1}\otimes D\otimes {\Id}^{\otimes m+n-i}\Big)\circ (\delta^{m-1}\otimes \Delta^{n-1})\circ\delta=m!n!D^{m,n}.$$
		Thus, we deduce that
\begin{eqnarray*}
			&&D^{m,n}(X)\\
			&=&\sum_{X}\sum_{i=1}^{m}\frac{1}{m!n!}(-1)^{\alpha_i}\pr^{0,1}_V(X_{(1)})\ldots  D^{0,1}(X_{(i)})\ldots \pr^{0,1}_V(X_{(m)})\otimes \pr^{0,1}_V(X_{(m+1)})\ldots \pr^{0,1}_V(X_{(m+n)})\\
			&&+\sum_{X}\sum_{j=1}^{n}\frac{1}{m!n!}(-1)^{\beta_j}\pr^{0,1}_V(X_{(1)})\ldots \pr^{0,1}_V(X_{(m)})\otimes \pr^{0,1}_V(X_{(m+1)})\ldots D^{0,1}(X_{(m+j)})\ldots \pr^{0,1}_V(X_{(m+n)}),
		\end{eqnarray*}
where $\alpha_i=D(X_{(1)}+\ldots +X_{(i-1)})$,~$\beta_j=D(X_{(1)}+\ldots +X_{(m+j-1)})$ and $(\delta^{m-1}\otimes \Delta^{n-1})\delta(X)=\sum_{X}X_{(1)}\otimes X_{(2)}\otimes\ldots\otimes X_{(m+n)},$ for all $X\in \huaComTrias^c(V).$ Therefore, the coderivation $D$ is completely determined by $D^{0,1}$. Moreover, there holds $D=\Psi(D^{0,1})$. On the one hand, this shows that the map $\Psi$ is well defined, and on the other hand, the map $\Psi$ is surjective. It is obvious that the map $\Psi$ is injective. Thus,  the map $\Psi$ is an isomorphism.
	\end{proof}
	
	Now we are ready to formulate the main result of this section.
	
	\begin{thm}\label{gla-post-lie}
		Let $V$ be a graded vector space. Then $(\Hom(\Sym(V)\otimes \bar{\Sym}(V),V),[\cdot,\cdot]_\postLie)$ is a graded Lie algebra, where the graded Lie bracket $[\cdot,\cdot]_\postLie$  is given by
		\begin{eqnarray*}
			[f,g]_\postLie:=f\circ g-(-1)^{kl}g\circ f,\,\,\,\,
			\label{eq:gfgcirc}
		\end{eqnarray*}
		for all $    f=\sum_{r\ge 0,s\ge 1} f_{r,s}\in \Hom^k(\Sym(V)\otimes \bar{\Sym}(V),V),
		g=\sum_{p\ge 0,q\ge 1} g_{p,q}\in \Hom^l(\Sym(V)\otimes \bar{\Sym}(V),V),$
		where $f\circ g\in \Hom^{k+l}(\Sym(V)\otimes \bar{\Sym}(V),V)$ is defined by
		{\small{
				\begin{eqnarray*}
					\nonumber&&(f\circ g)_{n,m}(x_1 \ldots  x_n\otimes x_{n+1} \ldots  x_{n+m})\\
					&=&\sum_{\sigma\in \mathbb S_{(p,q,n-p-q)}\atop 0\le p\le n-1,1\le q\le n}\varepsilon(\sigma)f_{n-p-q+1,m}\Big(g _{p,q}(x_{\sigma(1)} \ldots x_{\sigma(p)}\otimes x_{\sigma(p+1)} \ldots x_{\sigma(p+q)}) x_{\sigma(p+q+1)} \ldots  x_{\sigma(n)}\otimes  x_{n+1} \ldots x_{n+m}\Big)\\
					\nonumber&&+\sum_{\sigma\in\mathbb S_{(n-p,p)}\atop 0\le p\le n}\sum_{ \tau\in\mathbb S_{(q,m-q)}\atop 1\le q\le m}\varepsilon(\sigma)\varepsilon(\tau)(-1)^{l(x_{\sigma(1)}+\ldots+x_{\sigma(n-p)})}\\
					\nonumber&&f_{n-p,m-q+1}\Big(x_{\sigma(1)} \ldots  x_{\sigma(n-p)}\otimes  g_{p,q}(x_{\sigma(n-p+1)} \ldots x_{\sigma(n)}\otimes x_{n+\tau(1)} \ldots x_{n+\tau(q)}) x_{n+\tau(q+1)} \ldots  x_{n+\tau(m)}\Big).
				\end{eqnarray*}
		}}
	\end{thm}
	\begin{proof}
		It is obvious that $\coDer(\huaComTrias^c(V))$ is a graded Lie algebra, in which the graded Lie bracket $[\cdot,\cdot]_C$ is just the commutator. Transferring the graded Lie algebra structure from  $\big(\coDer(\huaComTrias^c(V)),[\cdot,\cdot]_C\big)$ to   $\Hom(\Sym(V)\otimes \bar{\Sym}(V),V)$ via the graded linear isomorphism  $\Psi$, we obtain a graded Lie algebra $(\Hom(\Sym(V)\otimes \bar{\Sym}(V),V),[\cdot,\cdot]_\postLie)$ as claimed.
	\end{proof}
	\begin{rmk}
		In fact, it is easy to see that $(\Hom(\Sym(V)\otimes \bar{\Sym}(V),V),\circ)$ is a graded pre-Lie algebra.
	\end{rmk}
\emptycomment{	\yh{\begin{rmk}
			Let $\huaP$ be a Koszul operad. The set of $\huaP_\infty$-algebra structures on a graded vector space $V$ is equivalently given by
			\begin{eqnarray*}
				\Hom_{\dgOp}(\Omega\huaP^{\mathrm{i}},\End_V)\cong \Tw(\huaP^{\mathrm{i}},\End_V)\cong\Hom_{\dgCoop}(\huaP^{\mathrm{i}},\B\End_V)\cong \Codiff(\huaP^{\mathrm{i}}(V[1])).
			\end{eqnarray*}
			We note that $\Tw(\huaP^{\mathrm{i}},\End_V)$ is the set of Maurer-Cartan elements of the graded Lie algebra $\Hom_{\mathbb S-\Mod}(\huaP^{\mathrm{i}},\End_V)$, which is the sub-adjacent Lie algebra of the graded  pre-Lie algebra of the convolution operad $\Hom_{\mathbb S-\Mod}(\huaP^{\mathrm{i}},\End_V)$. Since $\huaPostLie$ operad is a Koszul operad \cite{Val}, therefore $(\Hom(\Sym(V[1])\otimes \bar{\Sym}(V[1]),V[1]),\circ)$ is a graded pre-Lie algebra.
	\end{rmk}}
}
	An ordinary vector space $V$ can be naturally viewed as a graded vector space concentrated at degree $-1$. We will see in a moment that  in the corresponding graded Lie algebra obtained by Theorem \ref{gla-post-lie}  `controls' post-Lie algebra structures  on $V$.
	
	Let $V$ be a vector space. Consider the graded vector space
	$
	C_{\postLie}^*(V,V)=\oplus_{n=0}^{+\infty}C_{\postLie}^{n}(V,V),
	$
	 given by the formula
	$$
	C_{\postLie}^{n}(V,V)=\oplus_{i=0}^{n}\Hom(\wedge^{i}V\otimes \wedge^{n+1-i}V,V).
	$$
	We will write $f=(f_0,\ldots,f_n)\in C_{\postLie}^{n}(V,V)$, where $f_i\in \Hom(\wedge^iV\otimes\wedge^{n+1-i}V,V)$. For $f\in C_{\postLie}^{n}(V,V)$ and $g\in C_{\postLie}^{m}(V,V)$, define  $f\circ g=\big((f\circ g)_0,\ldots,(f\circ g)_{n+m}\big)\in C^{n+m}_\postLie(V,V)$ for $0\le k\le m$ by
	\begin{eqnarray*}
		&&(f\circ g)_k(x_1 \ldots x_{k}\otimes x_{k+1} \ldots x_{n+m+1})\\
		&=&\sum_{\sigma\in\mathbb S_{(k-j,j)}\atop 0\le j\le k}\sum_{ \tau\in\mathbb S_{(m+1-j,n+j-k)}}(-1)^\sigma(-1)^\tau(-1)^{m(k-j)}\\
		\nonumber&&f_{k-j}\big(x_{\sigma(1)} \ldots  x_{\sigma(k-j)}\otimes  g_{j}(x_{\sigma(k-j+1)} \ldots x_{\sigma(k)}\otimes x_{k+\tau(1)} \ldots x_{k+\tau(m+1-j)}) x_{k+\tau(m+2-j)} \ldots  x_{k+\tau(n+m+1-k)}\big),
	\end{eqnarray*}
	and for $m+1\le k\le m+n$ by
	\begin{eqnarray*}
		\nonumber&&(f\circ g)_{k}(x_1 \ldots  x_k\otimes x_{k+1} \ldots  x_{n+m+1})\\
		&=&\sum_{\sigma\in \mathbb S_{(j,m+1-j,k-m-1)}\atop 0\le j\le m}(-1)^{\sigma}f_{k-m}\big(g _{j}(x_{\sigma(1)} \ldots x_{\sigma(j)}\otimes x_{\sigma(j+1)} \ldots x_{\sigma(m+1)}) x_{\sigma(m+2)} \ldots  x_{\sigma(k)}\otimes  x_{k+1} \ldots x_{n+m+1}\big)\\
		\nonumber&&+\sum_{\sigma\in\mathbb S_{(k-j,j)}\atop 0\le j\le m}\sum_{ \tau\in\mathbb S_{(m+1-j,n+j-k)}}(-1)^\sigma(-1)^\tau(-1)^{m(k-j)}\\
		\nonumber&&f_{k-j}\big(x_{\sigma(1)} \ldots  x_{\sigma(k-j)}\otimes  g_{j}(x_{\sigma(k-j+1)} \ldots x_{\sigma(k)}\otimes x_{k+\tau(1)} \ldots x_{k+\tau(m+1-j)}) x_{k+\tau(m+2-j)} \ldots  x_{k+\tau(n+m+1-k)}\big).
	\end{eqnarray*}

	\begin{thm}\label{post-lie-gla}
		Let $V$ be a vector space. Then  $(C_{\postLie}^*(V,V),[\cdot,\cdot]_{\postLie}) $ is a graded Lie algebra, where the graded Lie bracket $[\cdot,\cdot]_{\postLie}$ is given by
		\begin{equation}
			[f,g]_{\postLie}:=f\circ g-(-1)^{nm}g\circ f,
			\quad \forall f\in C_{\postLie}^{n}(V,V),~g\in C_{\postLie}^{m}(V,V).
			\label{eq:glapL}
		\end{equation}
		Moreover, its Maurer-Cartan elements are precisely post-Lie algebra structures on the vector space  $V$.
	\end{thm}
	\begin{proof}
		
		The graded Lie algebra structure $[\cdot,\cdot]_{\postLie}$ is deduced from  Theorem \ref{gla-post-lie}.
		For all $\pi=(\pi_0,\pi_1)\in C_{\postLie}^{1}(V,V)$, where $\pi_0\in\Hom(\wedge^2V,V)$ and $\pi_1\in\Hom(V\otimes V,V)$, we have $$[\pi,\pi]_{\postLie}=2\pi\circ \pi=2((\pi\circ\pi)_0,(\pi\circ\pi)_1,(\pi\circ\pi)_2),$$ where the maps
		$$
		(\pi\circ\pi)_0\in\Hom(\wedge^3V,V),\quad (\pi\circ\pi)_1\in\Hom(V\otimes \wedge^2V,V),\quad (\pi\circ\pi)_2\in\Hom(\wedge^2V\otimes V,V)
		$$
		are given by
		\begin{eqnarray*}
			(\pi\circ\pi)_0(x_1 x_2 x_3)&=&\pi_0(\pi_0(x_1 x_2) x_3)-\pi_0(\pi_0(x_1 x_3) x_2)+\pi_0(\pi_0(x_2 x_3) x_1),\\
			(\pi\circ\pi)_1(x_1\otimes x_2x_3)&=&-\pi_1(x_1\otimes \pi_0(x_2 x_3))+\pi_0(\pi_1(x_1\otimes x_2) x_3)-\pi_0(\pi_1(x_1\otimes x_3) x_2),\\
			(\pi\circ\pi)_2(x_1 x_2\otimes x_3)&=&\pi_1(\pi_0(x_1 x_2)\otimes x_3)+\pi_1(\pi_1(x_1\otimes x_2)\otimes x_3)-\pi_1(\pi_1(x_2\otimes x_1)\otimes x_3)\\
			&&-\pi_1(x_1\otimes \pi_1(x_2\otimes x_3))+\pi_1(x_2\otimes \pi_1(x_1\otimes x_3)).
		\end{eqnarray*}
		Therefore, $\pi=(\pi_0,\pi_1)\in C_{\postLie}^{1}(V,V)$ is a Maurer-Cartan element of $(C_{\postLie}^*(V,V),[\cdot,\cdot]_{\postLie})$ if and only if   $(\pi\circ\pi)_0=0,~(\pi\circ\pi)_1=0,~(\pi\circ\pi)_2=0$.  Note that $(\pi\circ\pi)_0=0$ implies that $\pi_0$ is a Lie algebra structure on $V$, $(\pi\circ\pi)_1=0$ implies that $\pi_1(x_1,\cdot):V\to V$ is a derivation on the Lie algebra   $(V,\pi_0)$, and $(\pi\circ\pi)_2=0$ implies the compatibility condition \eqref{Post-2} for $\pi_0$ and $\pi_1$.
		Therefore, $\pi=(\pi_0,\pi_1)$ is a Maurer-Cartan element if and only if $(V,\pi_0,\pi_1)$ is a post-Lie algebra.
	\end{proof}
	
	Let $(\g,[\cdot,\cdot]_\g,\rhd)$ be a post-Lie algebra. Denote by $\pi=(\pi_0,\pi_1)$, where $\pi_0=[\cdot,\cdot]_\g$ and $\pi_1=\rhd$. By Theorem \ref{post-lie-gla},   $\pi$ is a Maurer-Cartan element of the graded Lie algebra   $(C_{\postLie}^*(\g,\g),[\cdot,\cdot]_{\postLie})$.
	Define the inner derivation $d_\pi:C_{\postLie}^*(\g,\g)\lon C_{\postLie}^*(\g,\g)$ by
	\begin{eqnarray*}
		d_\pi(f):=[\pi,f]_{\postLie},\,\,\,\,\forall f\in C_{\postLie}^*(\g,\g).
	\end{eqnarray*}
	By  $[\pi,\pi]_{\postLie}=0$, we obtain a
	dg Lie algebra $(C_{\postLie}^*(\g,\g),[\cdot,\cdot]_{\postLie},d_\pi)$. This dg Lie algebra governs deformations of the post-Lie algebra $(\g,[\cdot,\cdot]_\g,\rhd)$.

	\begin{thm}\label{def-post-lie}
		Let $(\g,\pi_0,\pi_1)$ be a post-Lie algebra. Then   $ (\g,\pi_0+\pi'_0,\pi_1+\pi'_1)$ is still a post-Lie algebra, where $\pi'_0:\g\wedge\g\longrightarrow \g$ and $\pi'_1:\g\otimes\g\longrightarrow \g$ are linear maps,   if and only if $\pi'=(\pi_0',\pi_1')$ is a Maurer-Cartan
		element of the dg Lie algebra
		$(C_{\postLie}^*(\g,\g),[\cdot,\cdot]_{\postLie},d_\pi)$.
	\end{thm}
	\begin{proof}
		It is straightforward.
	\end{proof}

	\section{Cohomologies of post-Lie algebras}\label{sec:coh}

	In this section, we define the cohomology of a  post-Lie algebra   using the dg Lie algebra given in Theorem \ref{def-post-lie}. As an application, we classify infinitesimal deformations of a post-Lie algebra using the second cohomology group of the relevant complex.
	
	Let $(\g,[\cdot,\cdot]_\g,\rhd)$ be a post-Lie algebra, and set $\pi:=(\pi_0,\pi_1)$, where $\pi_0=[\cdot,\cdot]_\g$ and $\pi_1=\rhd$. Define the set of  $0$-cochains $\frkC^0(\g;\g)$ to be $0$, and define
	the set of  $n$-cochains $\frkC^n(\g;\g)$ to be $C_{\postLie}^{n-1}(\g,\g)$ for $~n\geq 1$, i.e.
	$$
	\frkC^n(\g;\g)=\oplus_{i=0}^{n-1}\Hom(\wedge^{i}\g\otimes \wedge^{n-i}\g,\g).
	$$
	Define the {\bf coboundary operator} $\partial:\frkC^n(\g;\g)\lon \frkC^{n+1}(\g;\g)$ by
	\begin{equation}\label{cohomology-of-RB}
		\partial(f)=(-1)^{n-1}d_\pi(f)=(-1)^{n-1}[\pi,f]_{\postLie},\,\,\,\,\forall f\in \frkC^n(\g;\g).
	\end{equation}
	
	\begin{defi}\label{defi:cohomology of pL}
		The cohomology of the cochain complex $(\oplus _{n=0}^{+\infty}\frkC^n(\g;\g),\partial)$ is called the {\bf cohomology  of the post-Lie algebra} $(\g,[\cdot,\cdot]_\g,\rhd)$. We denote its $n$-th cohomology group by $\huaH^n(\g;\g)$.
	\end{defi}
	
	\begin{rmk}
		Since $\huaPostLie$ operad is a Koszul operad \cite{Val}, the operadic  cohomology theory of post-Lie algebras introduced here
		is isomorphic to the Andr\'e-Quillen cohomology \cite{LV,Mil}.
	\end{rmk}

	To better understand this cohomology, we write the differential $d_\pi$ more clearly.
	For any $f\in C_{\postLie}^{n-1}(\g,\g)$, we have
	\begin{eqnarray*}
		(d_\pi(f))_k(x_1 \ldots  x_k\otimes x_{k+1} \ldots  x_{n+1})=(\pi\circ f-(-1)^{n-1}f\circ \pi)_k(x_1 \ldots  x_k\otimes x_{k+1} \ldots  x_{n+1}).
	\end{eqnarray*}
	For $k=0$, we have
	\begin{eqnarray*}
		&&(\pi\circ f-(-1)^{n-1}f\circ \pi)_0(x_1 \ldots  x_{n+1})\\
		&=&\sum_{i=1}^{n+1}(-1)^{n+1-i}[f_0(x_1 \ldots \hat{x}_i\ldots  x_{n+1}),x_i]_\g\\
		&&-(-1)^{n-1}\sum_{1\le i<j\le n+1}(-1)^{i+j-1}f_0([x_i,x_j]_\g x_1\ldots \hat{x}_i\ldots  \hat{x}_j\ldots  x_{n+1}).
	\end{eqnarray*}
	For $k=1$, we have
	\begin{eqnarray*}
		&&(\pi\circ f-(-1)^{n-1}f\circ \pi)_1(x_1\otimes x_2 \ldots  x_{n+1})\\
		&=&(-1)^{n-1}x_1\rhd f_0(x_2 \ldots  x_{n+1})+\sum_{i=2}^{n+1}(-1)^{n+1-i}[f_1(x_1\otimes x_2 \ldots \hat{x}_i\ldots    x_{n+1}),x_i]_\g\\
		&&-(-1)^{n-1}\sum_{2\le i<j\le n+1}(-1)^{i+j}f_1(x_1\otimes [x_i,x_j]_\g x_2\ldots \hat{x}_i\ldots  \hat{x}_j\ldots  x_{n+1})\\
		&&-(-1)^{n-1}\sum_{i=2}^{n+1}(-1)^{i}f_0(x_1\rhd x_i x_2 \ldots \hat{x}_i\ldots    x_{n+1}).
	\end{eqnarray*}
	For $2\le k\le n-1$, we have
	\begin{eqnarray*}
		&&(\pi\circ f-(-1)^{n-1}f\circ \pi)_k(x_1 \ldots x_{k}\otimes x_{k+1} \ldots x_{n+1})\\
		&=&(-1)^{n-1}\sum_{i=1}^{k}(-1)^{i-1}x_i\rhd f_{k-1}(x_1\ldots \hat{x}_i\ldots x_k \otimes x_{k+1} \ldots x_{n+1})\\
		&&+\sum_{i=k+1}^{n+1}(-1)^{n+1-i}[f_k(x_1 \ldots x_{k}\otimes x_{k+1} \ldots \hat{x}_i\ldots x_{n+1}),x_i]_\g\\
		&&-(-1)^{n-1}\sum_{1\le i< j\le k}(-1)^{i+j-1}f_{k-1}([x_i,x_j]_\g x_1\ldots \hat{x}_i\ldots \hat{x}_j\ldots x_k\otimes x_{k+1} \ldots x_{n+1})\\
		&&-(-1)^{n-1}\sum_{1\le i< j\le k}(-1)^{i+j-1}f_{k-1}\big((x_i\rhd x_j-x_j\rhd x_i) x_1\ldots \hat{x}_i\ldots \hat{x}_j\ldots x_k\otimes x_{k+1} \ldots x_{n+1}\big)\\
		&&-(-1)^{n-1}\sum_{k+1\le i< j\le n+1}(-1)^{k+i+j-1}f_{k}(x_1 \ldots x_{k}\otimes [x_i,x_j]_\g x_{k+1}\ldots \hat{x}_i\ldots \hat{x}_j \ldots x_{n+1})\\
		&&-(-1)^{n-1}\sum_{i=1}^{k}\sum_{j=k+1}^{n+1}(-1)^{i+j+k}f_{k-1}(x_1 \ldots\hat{x}_i\ldots x_{k}\otimes x_i\rhd x_j x_{k+1}\ldots \hat{x}_j\ldots x_{n+1}).
	\end{eqnarray*}
	For $k=n$, we have
	\begin{eqnarray*}
		&&(\pi\circ f-(-1)^{n-1}f\circ \pi)_n(x_1 \ldots  x_{n}\otimes x_{n+1})\\
		&&=\sum_{j=0}^{n-1}\sum_{\sigma\in \mathbb S_{(j,n-j)}}(-1)^{\sigma}f_j(x_{\sigma(1)} \ldots x_{\sigma(j)}\otimes x_{\sigma(j+1)} \ldots x_{\sigma(n)})\rhd x_{n+1}\\
		&&+\sum_{i=1}^{n}(-1)^{n+i}x_i\rhd f_{n-1}(x_1 \ldots\hat{x}_i\ldots  x_{n}\otimes x_{n+1})\\
		&&-(-1)^{n-1}\sum_{1\le i<j\le n}(-1)^{i+j-1}f_{n-1}([x_i,x_j]_\g x_{1}\ldots \hat{x}_i\ldots \hat{x}_j \ldots x_{n}\otimes x_{n+1})\\
		&&-(-1)^{n-1}\sum_{1\le i<j\le n}(-1)^{i+j-1}f_{n-1}\big((x_i\rhd x_j-x_j\rhd x_i) x_{1}\ldots \hat{x}_i\ldots \hat{x}_j \ldots x_{n}\otimes x_{n+1}\big)\\
		&&-(-1)^{n-1}\sum_{i=1}^{n}(-1)^{i+1}f_{n-1}(x_{1}\ldots \hat{x}_i\ldots  x_{n}\otimes x_i\rhd x_{n+1}).
	\end{eqnarray*}

	We write $\partial =\sum_{k=0}^{n-1}\partial^k$, where $\partial^k=\partial|_{\Hom(\wedge^k\g\otimes \wedge^{n-k}\g,\g)}$. Then by the formula of $d_\pi$ given above, we can write
	$$
	\partial^{n-1}=\partial^{n-1}_{n-1}+\partial^{n-1}_n,\quad \partial^k=\partial^k_k+\partial^k_{k+1}+\partial^k_{n},\quad\mbox{for}~ k=0,\ldots,n-2,
	$$
	where $\partial^k_j$, $j\in\{k,k+1,n\}$ is the map $\Hom(\wedge^k\g\otimes \wedge^{n-k}\g,\g)\to\Hom(\wedge^j\g\otimes \wedge^{n+1-j}\g,\g)$ given by the following formulas:
	\begin{itemize}
		\item  for $k=0,1,\cdots,n-1$,
		\begin{eqnarray*}
			&&(\partial^k_{k}f_{k})(x_1\ldots x_{k}; x_{k+1} \ldots x_{n+1})\\
			&=&\sum_{i=k+1}^{n+1}(-1)^{i+1}[x_i,f_k(x_1 \ldots x_{k}; x_{k+1} \ldots \hat{x}_i\ldots x_{n+1})]_\g\\
			&&+\sum_{k+1\le i< j\le n+1}(-1)^{k+i+j}f_{k}(x_1 \ldots x_{k}; [x_i,x_j]_\g x_{k+1}\ldots,\hat{x}_i\ldots \hat{x}_j \ldots x_{n+1}),
		\end{eqnarray*}

		\item for $k=0,1,\cdots,n-2$,
		\begin{eqnarray*}
			&&(\partial_{k+1}^{k}f_{k})(x_1\ldots x_{k+1}; x_{k+2} \ldots x_{n+1})\\
			&=&\sum_{i=1}^{k+1}(-1)^{i-1}x_i\rhd f_{k}(x_1\ldots \hat{x}_i\ldots x_{k+1}; x_{k+2} \ldots x_{n+1})\\
			&&+\sum_{1\le i< j\le k+1}(-1)^{i+j}f_{k}\big(([x_i,x_j]_\g +x_i\rhd x_j-x_j\rhd x_i) x_1\ldots \hat{x}_i\ldots \hat{x}_j\ldots x_{k+1}; x_{k+2} \ldots x_{n+1}\big)\\
			&&-\sum_{i=1}^{k+1}\sum_{j=k+2}^{n+1}(-1)^{i+j+k+1}f_{k}(x_1 \ldots\hat{x}_i\ldots x_{k+1}; x_i\rhd x_j x_{k+2}\ldots \hat{x}_j\ldots x_{n+1}),
		\end{eqnarray*}

		\item for $k=0,1,\cdots,n-2$,
		
		\begin{eqnarray*}
			(\partial_{n}^{k} f_k)(x_1 \ldots  x_{n}; x_{n+1})=\sum_{\sigma\in \mathbb S_{(k,n-k)}}(-1)^{n-1}(-1)^{\sigma}f_k(x_{\sigma(1)} \ldots x_{\sigma(k)}; x_{\sigma(k+1)} \ldots x_{\sigma(n)})\rhd x_{n+1},
		\end{eqnarray*}
	\end{itemize}
	and
\begin{eqnarray*}
		&&(\partial_{n}^{n-1} f_{n-1})(x_1 \ldots  x_{n}; x_{n+1})\\
		&=&\sum_{i=1}^{n}(-1)^{i+1}x_i\rhd f_{n-1}(x_1 \ldots\hat{x}_i\ldots  x_{n}; x_{n+1})\\
		&&+\sum_{i=1}^{n}(-1)^{i+1}f_{n-1}(x_1\ldots\hat{x}_i\ldots x_{n};x_i)\rhd x_{n+1}\\
		&&+\sum_{1\le i<j\le n}(-1)^{i+j}f_{n-1}\big(([x_i,x_j]_\g +x_i\rhd x_j-x_j\rhd x_i) x_{1}\ldots \hat{x}_i\ldots \hat{x}_j \ldots x_{n}; x_{n+1}\big)\\
		&&-\sum_{i=1}^{n}(-1)^{i+1}f_{n-1}(x_{1}\ldots \hat{x}_i\ldots  x_{n}; x_i\rhd x_{n+1}).
	\end{eqnarray*}
We introduce the notations $C^{k,s}=\Hom(\wedge^k\g\otimes\wedge^{s}\g,\g)$.
This is	illustrated by the following diagram:
{\footnotesize
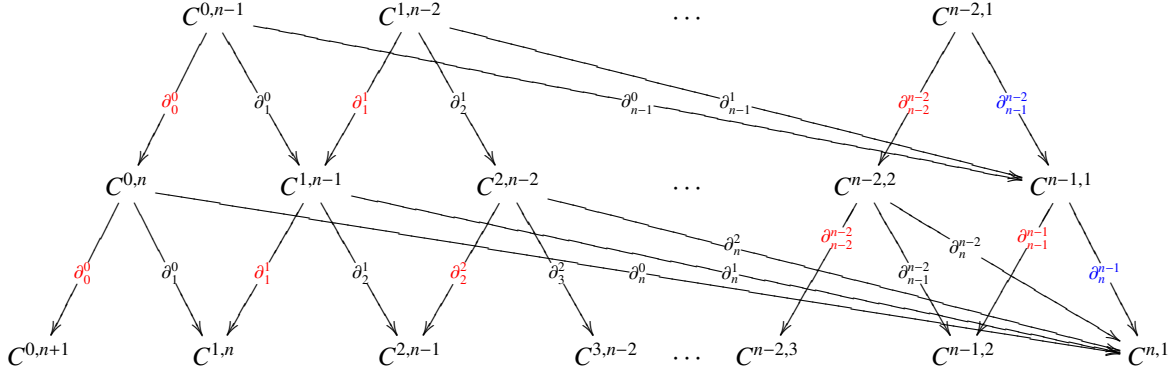
\begin{equation*}
\xymatrix@C=7pt@R=50pt{
&&C^{0,n-1}
\ar[dl]|{\color{red}{ {\partial}^0_0 }} \ar[dr]|{{\partial}^0_1} \ar[drrrrrrrrr]|{{\partial}^0_{n-1}} &
&C^{1,n-2}
\ar[dl]|{\color{red}{{\partial}^1_1}} \ar[dr]|{{\partial}^1_2} \ar[drrrrrrr]|{{\partial}^1_{n-1}}&
&&\cdots&
&&C^{n-2,1}
\ar[dl]|{\color{red}{{\partial}^{n-2}_{n-2}}} \ar[dr]| {\color{blue}{{\partial}^{n-2}_{n-1}}}&
&&\\
&C^{0,n}
\ar[dl]|{\color{red}{{\partial}^0_0 }} \ar[dr]|{{\partial}^0_1} \ar[drrrrrrrrrrr]|{{\partial}^0_n}&
&C^{1,n-1}
\ar[dl]|{\color{red}{{\partial}^1_1}} \ar[dr]|{{\partial}^1_2} \ar[drrrrrrrrr]|{{\partial}^1_n}&
&C^{2,n-2}
\ar[dl]|{\color{red}{{\partial}^2_2}} \ar[dr]|{{\partial}^2_3} \ar[drrrrrrr]|<<<<<<<<<<<<<<<<<<<<<<<<<{{\partial}^2_n}&
&\cdots&
&C^{n-2,2}
\ar[dl]|<<<<<<{\color{red}{{\partial}^{n-2}_{n-2}}} \ar[dr]|{{\partial}^{n-2}_{n-1}} \ar[drrr]|<<<<<<<<<<<<{{\partial}^{n-2}_{n}}&
&C^{n-1,1}
\ar[dl]|<<<<<<{\color{red}{{\partial}^{n-1}_{n-1}}} \ar[dr]|{\color{blue}{{\partial}^{n-1}_{n}}}&\\
C^{0,n+1}& & C^{1,n}& & C^{2,n-1}& & C^{3,n-2}& \cdots& C^{n-2,3}& & C^{n-1,2}& &C^{n,1}
}
\end{equation*}
}
	
	\begin{lem}\label{complex-lie}
		Let $(\g,[\cdot,\cdot]_\g,\rhd)$ be a post-Lie algebra. For all $k=0,1,2,\cdots$, the maps $\partial^k_k$ induce a differential on  $\oplus_{l=1}^{+\infty} \Hom(\wedge^{k}\g\otimes\wedge^l\g,\g)$, so that we have the following cochain complex:
		\begin{eqnarray*}
			\Hom(\wedge^{k}\g\otimes\g,\g)\stackrel{\partial^k_{k}}{\longrightarrow}\Hom(\wedge^{k}\g\otimes\wedge^{2}\g,\g)\stackrel{\partial^k_{k}}{\longrightarrow}\Hom(\wedge^{k}\g\otimes\wedge^{3}\g,\g)\stackrel{\partial^k_{k}}{\longrightarrow} \cdots.
		\end{eqnarray*}
	\end{lem}
	
	\begin{proof} We need to show that $\partial^k_k\circ \partial^k_k=0$.
		For any $f\in \Hom(\wedge^{k}\g\otimes\g,\g)$, we have
		\begin{eqnarray*}
			(\partial\circ \partial) f=\partial(\partial^{k}_{k}f+\partial^{k}_{k+1}f)=\partial^{k}_{k}\partial^{k}_{k}f+\partial^{k}_{k+1}\partial^{k}_{k}f+\partial^{k}_{k+2}\partial^{k}_{k}f+\partial^{k+1}_{k+1}\partial^{k}_{k+1}f+\partial^{k+1}_{k+2}\partial^{k}_{k+1}f=0.
		\end{eqnarray*}
		Thus, we deduce that $\partial^{k}_{k}\partial^{k}_{k}f=0,$ for all $f\in \Hom(\wedge^{k}\g\otimes\g,\g)$.
		
		For $m\ge 2$ and $g\in \Hom(\wedge^{k}\g\otimes\wedge^{m}\g,\g)$, we have
		\begin{eqnarray*}
			(\partial\circ \partial) g&=&\partial(\partial^{k}_{k}g+\partial^{k}_{k+1}g+\partial^{k}_{k+m}g)\\&=&\partial^{k}_{k}\partial^{k}_{k}g+\partial^{k}_{k+1}\partial^{k}_{k}g+\partial^{k}_{k+m+1}\partial^{k}_{k}g
+\partial^{k+1}_{k+1}\partial^{k}_{k+1}g+\partial^{k+1}_{k+2}\partial^{k}_{k+1}g+\partial^{k+1}_{k+m+1}\partial^{k}_{k+1}g\\
&&+\partial^{k+m}_{k+m}\partial^{k}_{k+m}g+\partial^{k+m}_{k+m+1}\partial^{k}_{k+m}g
			\\&=&0,
		\end{eqnarray*}
		which implies  that $\partial^{k}_{k}\partial^{k}_{k}g=0,$ for all $g\in \Hom(\wedge^{k}\g\otimes\wedge^{m}\g,\g)$.
	\end{proof}
	
	\begin{rmk}
		By the definition of $\partial^k_{k}$, obviously we have  $\partial^0_0=\dM_\CE$,  the Chevalley-Eilenberg coboundary operator of the Lie algebra $\g$ with coefficients in the adjoint representation. Thus the cochain complex $(\oplus_{l=1}^{+\infty} \Hom(\wedge^{k}\g\otimes\wedge^l\g,\g),\partial^k_k)$ can be viewed as  a higher version of the Chevalley-Eilenberg complex of the  Lie algebra $(\g,[\cdot,\cdot]_\g)$.
	\end{rmk}
	
	Now we have seen that the  Chevalley-Eilenberg complex of a Lie algebra is a subcomplex of the cochain complex associated to a post-Lie algebra defined above. This is reasonable and expected, since there is a Lie algebra inside a post-Lie algebra. Recall from \cite{TBGS} that there is another complex of a post-Lie algebra defined in terms of its sub-adjacent Lie algebra. Later in this paper, we indicate the precise relation between the cohomology given in \cite{TBGS} and the cohomology of a post-Lie algebra defined in Definition \ref{defi:cohomology of pL}.
	
	Note that $(\oplus_{n=1}^{+\infty}\Hom(\wedge^{n-1} \g\otimes \g,\g), \partial^{n-1}_n)$ is not a subcomplex. In fact, for all $f\in \Hom(\wedge^{n-1}\g\otimes\g,\g)$, we have
	\begin{eqnarray*}
		0=(\partial\circ \partial) f=\partial(\partial^{n-1}_{n-1}f+\partial^{n-1}_{n}f)=\partial^{n-1}_{n-1}\partial^{n-1}_{n-1}f+\partial^{n-1}_{n}\partial^{n-1}_{n-1}f
		+\partial^{n-1}_{n+1}\partial^{n-1}_{n-1}f+\partial^{n}_{n}\partial^{n-1}_{n}f
		+\partial^{n}_{n+1}\partial^{n-1}_{n}f,
	\end{eqnarray*}
	which implies that $$\partial^{n}_{n+1}\partial^{n-1}_{n}f+\partial^{n-1}_{n+1}\partial^{n-1}_{n-1}f=0,\quad\forall f\in \Hom(\wedge^{n-1}\g\otimes\g,\g).$$
	Thus, the obstruction of $\partial^{n}_{n+1}\partial^{n-1}_{n}f$ being zero is given by $\partial^{n-1}_{n+1}\partial^{n-1}_{n-1}f$. So one natural approach to obtain a reduced complex is considering the graded space $\oplus_{n=1}^{+\infty}\ker(\partial^{n-1}_{n-1})\subset \oplus_{n=1}^{+\infty}\Hom(\wedge^{n-1} \g\otimes \g,\g)$. By the explicit formula of $\partial^{n-1}_{n-1}$, we have
	\begin{eqnarray*}
		&&(\partial^{n-1}_{n-1}f)(x_1\ldots x_{n-1}; x_{n} x_{n+1})\\
		&=&(-1)^{n+1}[x_{n},f(x_1\ldots x_{n-1}; x_{n+1})]_\g+(-1)^{n+2}[x_{n+1},f(x_1\ldots x_{n-1}; x_{n})]_\g\\
		&&+(-1)^{n+2}f(x_1 \ldots x_{n-1}; [x_{n},x_{n+1}]_\g).
	\end{eqnarray*}
	In \cite{TBGS}, $\ker(\partial^{n-1}_{n-1})$ is denoted by $C_{\Der}^{n}(\g;\g)$.  Now we obtain a reduced complex as follows.

	\begin{lem}\label{lem:reduced complex}
		Let $(\g,[\cdot,\cdot]_\g,\rhd)$ be a post-Lie algebra. Then $\partial^{*-1}_*$ determines a differential on  $\oplus_{n=1}^{+\infty}\Der^n(\g,\g)$ so that we have a cochain complex as follows:
		\begin{eqnarray*}
			C_{\Der}^{1}(\g;\g)\stackrel{\partial^0_{1}}{\longrightarrow}C_{\Der}^{2}(\g;\g)\stackrel{\partial^1_{2}}{\longrightarrow}C_{\Der}^{3}(\g;\g)\stackrel{\partial^2_{3}}{\longrightarrow} \cdots.
		\end{eqnarray*}
	\end{lem}

	\begin{rmk}
		The cochain complex given in Lemma \ref{lem:reduced complex} is exactly the one  given in \cite[Theorem 4.9]{TBGS}, whose cohomology are used to classify certain operator homotopy post-Lie algebras.
	\end{rmk}
	
	\begin{rmk}
		Besides the higher versions of the Chevalley-Eilenberg cochain complex given in  Lemma \ref{complex-lie} and the reduced cochain complex of  Lemma \ref{lem:reduced complex}, there are other cochain complexes inside the cochain complex of a post-Lie algebra. Similarly to the proof of Lemma \ref{complex-lie}, it is straightforward to obtain that $(\oplus_{n=1}^{+\infty}\Hom(\wedge^{n-1} \g\otimes \wedge^k\g,\g), \partial^{*-1}_*)$ is a cochain complex for all $k=3,4,\cdots.$
	\end{rmk}

	Next we give applications of the first and the second cohomology group.

 Note that $f\in \frkC^1(\g;\g)$  is closed  if and only if
	\begin{eqnarray}
		\partial(f)_0(x_1x_2)&=&[x_1,f(x_2)]_\g+[f(x_1),x_2]_\g-f([x_1,x_2]_\g)=0,\\
		\partial(f)_1(x_1\otimes x_2)&=&f(x_1)\rhd x_2+x_1\rhd f(x_2)-f(x_1\rhd x_2)=0.
	\end{eqnarray}
	Therefore, we have the following proposition.
	
	\begin{pro}
		A $1$-cochain  $ f\in \frkC^1(\g;\g)$ is a $1$-cocycle if and only $f$ is a derivation. Moreover,  $\huaH^1(\g;\g)=\Der(\g)$.
	\end{pro}
	
	Now we introduce the notion of an $\R$-deformation of a post-Lie algebra, where $\R$ is a local pro-Artinian $\K$-algebra. Since $\R$ is the projective limit of commutative local Artinian $\K$-algebras, $\R$ is equipped with an augmentation $\epsilon:\R\lon \K$. See \cite{Ma2} for more details about $\R$-deformation theory of algebraic structures. An infinitesimal deformation corresponds to the case $\R=\K[t]/(t^{2})$.
	
	Replacing the $\K$-vector spaces and $\K$-linear maps by $\R$-modules and $\R$-linear maps in Definition \ref{post-lie-defi} and Definition \ref{post-lie-homo-defi}, we obtain the definitions of $\R$-post-Lie algebras and homomorphisms between them.
	
	Any post-Lie algebra  $(\g,[\cdot,\cdot]_\g,\rhd_\g)$ can be viewed as an $\R$-post-Lie algebra    with the help of the augmentation map $\epsilon$. More precisely, the $\R$-module structure on $\g$ is given by
	$$
	r\cdot x:=\epsilon(r)x,\quad \forall r\in\R,~x\in\g.
	$$
	\begin{defi}
		An {\bf $\R$-deformation} of a $\K$-post-Lie algebra $(\g,[\cdot,\cdot]_\g,\rhd_\g)$ consists of an $\R$-post-Lie algebra structure $[\cdot,\cdot]_\R$ on the completed  tensor product $\R\hat{\otimes}\g$ such that $\epsilon\hat{\otimes}\Id_\g$ is an $\R$-post-Lie algebra homomorphism from $(\R\hat{\otimes}\g,[\cdot,\cdot]_\R,\rhd_\R)$ to $(\g,[\cdot,\cdot]_\g,\rhd_\g)$.
	\end{defi}
	
	We denote an $\R$-deformation of  $(\g,[\cdot,\cdot]_\g,\rhd_\g)$ by a triple $(\R\hat{\otimes}\g,[\cdot,\cdot]_\R,\rhd_\R)$.
	There is also a natural notion of equivalence between   $\R$-deformations.
	\begin{defi}
		Let $(\R\hat{\otimes}\g,[\cdot,\cdot]_\R,\rhd_\R)$ and $(\R\hat{\otimes}\g,[\cdot,\cdot]_\R',\rhd_\R')$ be two $\R$-deformations of a post-Lie algebra $(\g,[\cdot,\cdot]_\g,\rhd_\g)$. We call them {\bf equivalent} if there exists an $\R$-post-Lie algebra isomorphism $\phi:(\R\hat{\otimes}\g,[\cdot,\cdot]_\R',\rhd_\R')\lon(\R\hat{\otimes}\g,[\cdot,\cdot]_\R,\rhd_\R)$
		such that
		\begin{eqnarray}\label{equivalent-deformation-2}
			\epsilon\hat{\otimes}\Id_\g=(\epsilon\hat{\otimes}\Id_\g)\circ\phi.
		\end{eqnarray}
	\end{defi}
	
	\begin{rmk}
		Let $(\g,[\cdot,\cdot]_\g,\rhd_\g)$ be a post-Lie algebra. We have the following  {\bf deformation functor} $$\Def_\g:\pcDGAloc\lon \Set\footnote{$\pcDGAloc$  is the category of commutative local pro-Artinian $\K$-algebras and $\Set$ is the category of sets.}$$ which is associated to the equivalent classes of deformations of the post-Lie algebra $(\g,[\cdot,\cdot]_\g,\rhd_\g)$. Moreover,  the deformation functor $\Def_\g$ is isomorphic to
		the MC moduli functor $\MCmodu({C_{\postLie}^*(\g,\g)},-)$, where $C_{\postLie}^*(\g,\g)$   is   the dg Lie algebra given in Theorem \ref{def-post-lie}.
	\end{rmk}
	
	We will now introduce infinitesimal deformations and see that they are classified by the second cohomology group.
	\begin{defi}
		A $\K[t]/(t^{2})$-deformation of the  $\K$-post-Lie algebra $(\g,[\cdot,\cdot]_\g,\rhd_\g)$ is called an {\bf infinitesimal deformation}.
	\end{defi}
	
	Note that $\K[t]/(t^{2})\hat{\otimes}\g=\K[t]/(t^{2})\otimes\g$.  Let $\R=\K[t]/(t^{2})$ and $(\R\otimes\g,[\cdot,\cdot]_\R,\rhd_\R)$ be an infinitesimal deformation of $(\g,[\cdot,\cdot]_\g,\rhd_\g)$. Since $(\R\otimes\g,[\cdot,\cdot]_\R,\rhd_\R)$ is a post-Lie algebra, there exist $\pi_0,\omega_0\in\Hom(\g\wedge\g,\g)$ and $\pi_1,\omega_1\in\Hom(\g\otimes\g,\g)$ such that
	\begin{eqnarray}\label{rota-baxter-infinitesimal-deformation}
		[\cdot,\cdot]_\R=\pi_0+t\omega_0,\quad \rhd_\R=\pi_1+t\omega_1.
	\end{eqnarray}
	Since $\epsilon\otimes\Id_\g$ is an $\R$-post-Lie algebra homomorphism from $(\R\otimes\g,[\cdot,\cdot]_\R,\rhd_\R)$ to $(\g,[\cdot,\cdot]_\g,\rhd_\g)$, we deduce that $\pi_0=[\cdot,\cdot]_\g$ and $\pi_1=\rhd_\g$. Therefore, an infinitesimal deformation of $(\g,[\cdot,\cdot]_\g,\rhd_\g)$  is determined by a pair $(\omega_0,\omega_1)$.
	By the fact that $(\R\otimes\g,[\cdot,\cdot]_\g+t\omega_0)$ is an $\R$-Lie algebra, we get
	\begin{eqnarray}
		\label{post-def-1} \dM_\CE \omega_0&=&0.
	\end{eqnarray}
	Then by \eqref{Post-1}, we deduce that
	\begin{eqnarray}
		\label{post-def-2}&&\omega_1(x,[y,z]_\g)+x\rhd_\g\omega_0(y,z)-\omega_0(x\rhd_\g y,z)\\
		\nonumber&&-[\omega_1(x,y),z]_\g-\omega_0(y,x\rhd_\g z)-[y,\omega_1(x,z)]_\g=0.
	\end{eqnarray}
	Finally by \eqref{Post-2}, we obtain that
	\begin{eqnarray}
		\label{post-def-3}&&\big(\omega_1(x,y)-\omega_1(y,x)+\omega_0(x,y)\big)\rhd_\g z+\omega_1\big((x\rhd_\g y-y\rhd_\g x+[x,y]_\g),z\big)\\
		\nonumber&&-\omega_1(x,y\rhd_\g z)-x\rhd_\g \omega_1(y,z)+\omega_1(y,x\rhd_\g z)+y\rhd_\g \omega_1(x,z)=0.
	\end{eqnarray}
	\begin{thm}
		The pair $(\omega_0,\omega_1)$ determines an  infinitesimal  deformation  of the post-Lie algebra $(\g,[\cdot,\cdot]_\g,\rhd_\g)$ if and only if  $(\omega_0,\omega_1)$ is a $2$-cocycle of the post-Lie algebra $(\g,[\cdot,\cdot]_\g,\rhd_\g)$.
		
		Moreover, there is a one-to-one correspondence between   equivalence classes of infinitesimal deformations of the post-Lie algebra $(\g,[\cdot,\cdot]_\g,\rhd_\g)$ and the second cohomology group $\huaH^2(\g;\g)$.
	\end{thm}
	\begin{proof}
		By \eqref{post-def-1}, \eqref{post-def-2} and \eqref{post-def-3},    $(\omega_0,\omega_1)$ determines an  infinitesimal  deformation  of the post-Lie algebra $(\g,[\cdot,\cdot]_\g,\rhd_\g)$ if and only if $(\omega_0,\omega_1)$ is a $2$-cocycle.

		If two infinitesimal deformations   determined by $(\omega_0,\omega_1)$ and $(\omega_0',\omega_1')$ are equivalent, then there exists an $\R$-post-Lie algebra isomorphism $\phi$ from $(\R\otimes\g,[\cdot,\cdot]_\g+t\omega_0',\rhd_\g+t\omega_1')$ to $(\R\otimes\g,[\cdot,\cdot]_\g+t\omega_0,\rhd_\g+t\omega_1)$. By \eqref{equivalent-deformation-2}, we deduce that
		\begin{eqnarray}
			\phi={\Id_\g}+tf,\quad \mbox{where}\quad f\in\Hom(\g,\g).
		\end{eqnarray}
		
		Since ${\Id_\g}+tf$ is an  isomorphism from $(\R\otimes\g,[\cdot,\cdot]_\g+t\omega_0')$ to $(\R\otimes\g,[\cdot,\cdot]_\g+t\omega_0)$, we get
		\begin{equation}
			\label{eq:equmor1} \omega_0'-\omega_0=\dM_\CE f.
		\end{equation}
		By the equality
		$$
		({\Id_\g}+tf)\big(x\rhd_\g y+t\omega_1'(x,y)\big)=\big(({\Id_\g}+tf)x\big)\rhd_\g\big(({\Id_\g}+tf)y\big)+t\omega_1\big(({\Id_\g}+tf)x,({\Id_\g}+tf)y\big),
		$$
		we deduce that
		\begin{equation}
			\label{eq:equmor2}\omega_1'(x,y)-\omega_1(x,y)=f(x)\rhd_\g y+x\rhd_\g f(y)-f(x\rhd_\g y).
		\end{equation}
		
		By \eqref{eq:equmor1} and \eqref{eq:equmor2}, we deduce that
		$
		(\omega_0',\omega_1')-(\omega_0,\omega_1)=\partial(f),
		$
		which implies that $(\omega_0',\omega_1')$ and $(\omega_0,\omega_1)$ are in the same cohomology class if and only if the corresponding infinitesimal deformations of $(\g,[\cdot,\cdot]_\g,\rhd_\g)$ are equivalent.
	\end{proof}


	\section{Post-Lie$_\infty$ algebras}\label{sec:hom}

	In this section, we introduce the notion  of post-Lie$_\infty$ algebras using the graded Lie algebra given in Theorem \ref{gla-post-lie}. There is an equivalent characterization of post-Lie$_\infty$ algebras via open-closed homotopy Lie algebras. Various examples of post-Lie$_\infty$ algebras are constructed from higher geometric structures, e.g. homotopy Poisson structures and $L_\infty$-algebroids.

	\subsection{Post-Lie$_\infty$ algebras and  open-closed homotopy Lie algebras}
	
	Recall that for a graded vector space $V$ there is a symmetric coalgebra $\Sym(V)$ (endowed with the cocommutative shuffle coproduct).
	A linear map $D:\Sym(V)\lon \Sym(V)$ of degree $n$ is called a {\bf coderivation of degree $n$} on the coalgebra $(\Sym(V),\Delta,\varepsilon)$ if
	\begin{eqnarray}\label{coder}
		\Delta\circ D=(D\otimes {\Id}+{\Id}\otimes D)\circ \Delta.
	\end{eqnarray}
	Denote the vector space of coderivations of degree $n$ by $\coDer^n(\Sym(V))$.  Then $\coDer(\Sym(V)):=\oplus_{n\in\mathbb Z}\coDer^n(\Sym(V))$ has a graded Lie algebra structure as follows:
	\begin{eqnarray*}
		[D_1,D_2]_C:=D_1\circ D_2-(-1)^{D_1D_2}D_2\circ D_1,\,\,\forall D_1,D_2\in \coDer(\Sym(V)).
	\end{eqnarray*}

	Denote by $\Hom^n(\Sym(V),V)$ the space of total degree $n$ linear maps from  $\Sym(V)$ to $V$. An element $f\in\Hom^n(\Sym(V),V)$ is the sum of its components $f_i:\Sym^i(V)\lon V$ so that  $f=\sum_{i=0}^{+\infty} f_i$. Moreover, we can write any homomorphism $S(V)\to V$ as a sum of its graded components: $$\Hom(\Sym(V),V)=\oplus_{n\in\mathbb Z}\Hom^n(\Sym(V),V).$$
	Since $(\Sym(V),\Delta,\varepsilon)$ is the cofree object in the category of conilpotent cocommutative coalgebras, we deduce that $\coDer(\Sym(V))\cong \Hom(\Sym(V),V)$.  	More precisely, we have the  isomorphism $\Psi:\Hom(\Sym(V),V)\lon \coDer(\Sym(V))$ of graded vector spaces, which is given by
	\begin{eqnarray}
		\label{coderivation}&&\Psi(f)(v_1 \ldots  v_n)=\sum_{i=0}^{n}\sum_{\sigma\in\mathbb S_{(i,n-i)}}\varepsilon(\sigma;v_1,\ldots,v_n)f_i\big(v_{\sigma(1)},\ldots,v_{\sigma(i)}\big)  v_{\sigma(i+1)} \ldots  v_{\sigma(n)},
	\end{eqnarray}
	here $n=0,1,\ldots,$ and $v_1,\ldots,v_n\in V.$ We will denote $\Psi(f)$ by $\hat{f}$.
	
	Transferring  the graded Lie algebra structure $\big(\coDer(\Sym(V)),[\cdot,\cdot]_C\big)$ to   $\Hom(\Sym(V),V)$ via the graded linear isomorphism  $\Psi$, we obtain a graded Lie algebra  $(\Hom(\Sym(V),V),[\cdot,\cdot]_{\NR})$. The graded Lie bracket $[\cdot,\cdot]_{\NR}$, the {\bf Nijenhuis-Richardson bracket},  is given by the formula
	\begin{eqnarray*}
		[f,g]_{\NR}:=f\circ g-(-1)^{mn}g\circ f,\,\,\,\,\forall f=\sum_{i=0}^{+\infty} f_i\in \Hom^m(\Sym(V),V),~g=\sum_{j=0}^{+\infty} g_j\in \Hom^n(\Sym(V),V),
	\end{eqnarray*}
	where $f\circ g\in \Hom^{m+n}(\Sym(V),V)$ is defined by
	\begin{eqnarray*}
		f\circ g&=&\Big(\sum_{i=0}^{+\infty}f_i\Big)\circ\Big(\sum_{j=0}^{+\infty}g_j\Big):=\sum_{k=0}^{+\infty}\Big(\sum_{i+j=k+1}f_i\circ g_j\Big),
	\end{eqnarray*}
	while $f_i\circ g_j\in \Hom(\Sym^{i+j-1}(V),V)$ is defined by
	\begin{eqnarray*}
		(f_i\circ g_j)(v_1\ldots v_{i+j-1})
		:=\sum_{\sigma\in\mathbb S_{(j,i-1)}}\varepsilon(\sigma)f_i(g_j(v_{\sigma(1)}\ldots v_{\sigma(j)}) v_{\sigma(j+1)} \ldots v_{\sigma(i+j-1)}).
	\end{eqnarray*}
	
	It is clear that  $(\Hom(\overline{\Sym}(V),V),[\cdot,\cdot]_{\NR})$ is a graded Lie subalgebra of  $(\Hom(\Sym(V),V),[\cdot,\cdot]_{\NR})$.

	\begin{defi}\label{graded-Nijenhuis-Richardson-bracket}{\rm (\cite{LS})}
		An {\bf $ L_\infty$-algebra} structure  on a graded vector space $\g$ is a Maurer-Cartan element  $\sum_{i=1}^{+\infty}l_i$ of the graded Lie algebra $(\Hom(\overline{\Sym}(\g),\g),[\cdot,\cdot]_{\NR})$. More precisely, an $L_\infty$-algebra  is a $\mathbb Z$-graded vector space $\g=\oplus_{k\in\mathbb Z}\g^k$ equipped with a collection $(k\ge 1)$ of linear maps $l_k:\otimes^k\g\lon\g$ of degree $1$ with the property that, for any homogeneous elements $x_1,\ldots,x_n\in \g$, we have
		\begin{itemize}\item[\rm(i)]
			{\em (graded symmetry)} for every $\sigma\in\mathbb S_{n}$,
			\begin{eqnarray*}
				l_n(x_{\sigma(1)},\ldots,x_{\sigma(n-1)},x_{\sigma(n)})=\varepsilon(\sigma)l_n(x_1,\ldots,x_{n-1},x_n),
			\end{eqnarray*}
			\item[\rm(ii)] {\em (generalized Jacobi identity)} for all $n\ge 0$,
			\begin{eqnarray*}\label{sh-Lie}
				\sum_{i=1}^{n}\sum_{\sigma\in \mathbb S_{(i,n-i)} }\varepsilon(\sigma)l_{n-i+1}(l_i(x_{\sigma(1)},\ldots,x_{\sigma(i)}),x_{\sigma(i+1)},\ldots,x_{\sigma(n)})=0.
			\end{eqnarray*}
		\end{itemize}
	\end{defi}
	We will write $(\g,\{l_k\}_{k=1}^{+\infty})$ for a graded vector space $\g$ together with an $L_\infty$-algebra structure on it.   
	
	There is a canonical way to view a dg Lie algebra as an $L_\infty$-algebra.
	
	\begin{lem}\label{Quillen-construction}
		Let $(\g,[\cdot,\cdot]_\g,d)$ be a dg Lie algebra. Then  $(s^{-1}\g,\{l_i\}_{i=1}^{+\infty})$ is an $L_\infty$-algebra, where $
		l_1(s^{-1}x)=-s^{-1}d(x),~
		l_2(s^{-1}x,s^{-1}y)=(-1)^{x}s^{-1}[x,y]_\g,~
		l_k=0,$ for all $k\ge 3,
		$
		and homogeneous elements $x,y\in \g$.
	\end{lem}

	Let $(\g,\{l_k\}_{k=1}^{+\infty})$ be an $L_\infty$-algebra. Then $(\Sym(\g),\Delta,\Psi(\sum_{i=1}^{+\infty}l_i))$ is a dg coalgebra. Homomorphisms between $L_\infty$-algebras can be defined in term of  corresponding dg coalgebra structures.

	\begin{defi}{\rm (\cite{KS})}
		Let $(\g,\{l_k\}_{k=1}^{+\infty})$ and $(\g',\{l_k'\}_{k=1}^{+\infty})$  be  two $L_\infty$-algebras.
		An $L_\infty$-algebra {\bf homomorphism} from $(\g,\{l_k\}_{k=1}^{+\infty})$ to $(\g',\{l_k'\}_{k=1}^{+\infty})$ consists of a collection of degree $0$ graded multilinear maps $f_k:\g^{\otimes k}\lon \g',~ k\ge 1$ with the property
		$f_n(v_{\sigma(1)},\ldots,v_{\sigma(n)})
		=\varepsilon(\sigma)f_n(v_1,\ldots,v_n),
		$ for any $n\geq 1$ and homogeneous elements $v_1,\ldots,v_n\in \g$,
		and
		\begin{eqnarray}
			\nonumber&&\sum_{i=1}^n\sum_{\sigma\in \mathbb S_{(i,n-i)} }\varepsilon(\sigma)f_{n-i+1}\Big(l_i(v_{\sigma(1)},\ldots,v_{\sigma(i)}),v_{\sigma(i+1)},\ldots,v_{\sigma(n)}\Big)\\
			\label{L-infty-homo}&=&\sum_{i=1}^n\sum_{k_1+\ldots+k_i=n}\sum_{\sigma\in \mathbb S_{(k_1,\ldots,k_i)}}\frac{\varepsilon(\sigma)}{i!}l_i'\Big(f_{k_1}(v_{\sigma(1)},\ldots,v_{\sigma(k_1)}),\ldots,f_{k_i}(v_{\sigma(k_1+\ldots+k_{i-1}+1)},\ldots,v_{\sigma(n)})\Big).
		\end{eqnarray}
		Equivalently, $\bar{f}:\Sym(\g)\to \Sym(\g')$ is a homomorphism of dg coalgebras, where $f=\sum_{k=1}^{+\infty}f_k$ and $\bar{f}$ is given by \eqref{co-alg-homo}. 
	\end{defi}

	Now we are ready to introduce the notion of a post-Lie$_\infty$ algebra using the graded Lie algebra given in Theorem  \ref{gla-post-lie}.
	\begin{defi}\label{defi:homtopy-post-lie}
		A {\bf post-Lie$_\infty$ algebra} structure on a graded vector space $\g$ is defined to be a Maurer-Cartan element $\sum_{p\ge 0,q\ge 1}\frkM_{p,q}$ of the graded Lie algebra $(\Hom(\Sym(\g)\otimes \bar{\Sym}(\g),\g),[\cdot,\cdot]_\postLie)$. More precisely,
		a  post-Lie$_\infty$ algebra is a $\mathbb Z$-graded vector space $\g=\oplus_{k\in\mathbb Z}\g^k$ equipped with a collection of linear maps $\frkM_{p,q}:\Sym^p(\g)\otimes \Sym^q(\g)\lon\g$ of degree $1$ for $p\ge 0,q\ge1$ satisfying, for every collection of
		homogeneous elements $x_1,\ldots,x_{m+n}\in \g$, the following identities:
		{\small{
				\begin{eqnarray}
					\label{homotopy-post-lie}&&\sum_{\sigma\in \mathbb S_{(j,i,n-j-i)}\atop 0\le j\le n-1, 1\le i\le n}\varepsilon(\sigma)\frkM_{n-j-i+1,m}\Big(\frkM_{j,i}(x_{\sigma(1)}\ldots x_{\sigma(j)}\otimes x_{\sigma(j+1)}\ldots x_{\sigma(j+i)})x_{\sigma(j+i+1)}\ldots x_{\sigma(n)}\otimes x_{n+1}\ldots x_{n+m}\Big)\\
					\nonumber&&+\sum_{\sigma\in\mathbb S_{(n-j,j)}\atop 0\le j\le n}\sum_{ \tau\in\mathbb S_{(i,m-i)}\atop 1\le i\le m}\varepsilon(\sigma)\varepsilon(\tau)(-1)^{x_{\sigma(1)}+\ldots+x_{\sigma(n-j)}}\\
					\nonumber&&\frkM_{n-j,m-i+1}\Big(x_{\sigma(1)} \ldots x_{\sigma(n-j)}\otimes \frkM_{j,i}\big(x_{\sigma(n-j+1)} \ldots x_{\sigma(n)}\otimes x_{n+\tau(1)} \ldots x_{n+\tau(i)}\big) x_{n+\tau(i+1)} \ldots x_{n+\tau(m)}\Big)\\
					\nonumber&&=0,
				\end{eqnarray}
		}}
		for $n\ge 0,m\ge1.$
	\end{defi}
	
	Several remarks are in order to illustrate the relation between post-Lie$_\infty$ algebras, $L_\infty$-algebras, pre-Lie$_\infty$ algebras and operator homotopy post-Lie algebras.

	\begin{rmk}
		Let $(\g,\{\frkM_{p,q}\}_{p\ge0,q\ge1})$ be a post-Lie$_\infty$ algebra. Setting $n=0$ in \eqref{homotopy-post-lie}, we obtain that $\{\frkM_{0,q}\}_{q\ge1}$ is an {\bf $L_\infty$-algebra} structure on the graded vector space $\g$.
	\end{rmk}

	\begin{rmk}
		Let $(\g,\{\frkM_{p,q}\}_{p\ge0,q\ge1})$ be a post-Lie$_\infty$ algebra that satisfies $\frkM_{p,q}=0,q\ge 2$. For all $n\ge 1$, we define $\oprn_n:\otimes^n \g\lon \g, n\ge 1,$ by 
		\begin{eqnarray}\label{homotopy-prelie}
			\oprn_n(x_1,\ldots,x_n)\triangleq\frkM_{n-1,1}(x_1 \ldots x_{n-1}\otimes x_n).
		\end{eqnarray}
		Then $(\g,\{\oprn_n\}_{n\ge1})$ is a {\bf pre-Lie$_\infty$ algebra}, which was introduced in \cite{CL}. See \cite{Mer} for more applications of pre-Lie$_\infty$ algebras in geometry.
	\end{rmk}

	\begin{rmk}
		Let $(\g,\{\frkM_{p,q}\}_{p\ge0,q\ge1})$ be a post-Lie$_\infty$ algebra that satisfies $\frkM_{0,q}=0,q\ge 3$ and $\frkM_{p,q}=0,p\ge 1,q\ge 2$. We define $[\cdot,\cdot]_\g:\g\otimes\g\lon\g$ by
		$
			[x,y]_\g\triangleq\frkM_{0,2}(x y), $ for all $x,y\in\g.
		$
		For all $n\ge 1$, we define $\oprn_n:\otimes^n \g\lon \g, n\ge 1,$  by
		\begin{eqnarray}\label{homotopy-postlie-2}
			\oprn_n(x_1,\ldots,x_n)\triangleq\frkM_{n-1,1}(x_1 \ldots x_{n-1}\otimes x_n).
		\end{eqnarray}
		Then $(\g,[\cdot,\cdot]_\g,\{\oprn_n\}_{n\ge1})$ is an  {\bf\opt homotopy post-Lie algebra}, which was introduced in \cite{TBGS} and the first attempt to study the homotopy theory of post-Lie algebras.
	\end{rmk}
	
Now we recall the notion of an  open-closed homotopy Lie algebra, and use open-closed homotopy Lie algebras to characterize post-Lie$_\infty$ algebras.	
	\begin{defi}{\rm (\cite{Lazarev,MZ,Mer-1})}
		An  {\bf open-closed homotopy Lie algebra} $(\g,\h,\{l_k\}_{k=1}^{+\infty},\{\frkR_{p,q}\}_{p+q\ge1})$  consists of an $L_\infty$-algebra $(\g,\{l_k\}_{k=1}^{+\infty})$ and
		a family of graded linear maps $\frkR_{p,q}:\Sym^p(\g)\otimes \Sym^q(\h)\lon\h$ of degree $1$ for $p\ge 0,q\ge0$ with $p+q\ge1$ satisfying the following compatibility conditions:
		\emptycomment{
			\begin{eqnarray}
				\nonumber&&\sum_{i=1}^{n}\sum_{\sigma\in \mathbb S_{(i,n-i)}}\varepsilon(\sigma)\frkR_{n-i+1,m}\Big(l_i(x_{\sigma(1)},\ldots,x_{\sigma(i)}),x_{\sigma(i+1)},\ldots,x_{\sigma(n)};u_1,\ldots,u_m\Big)\\
				\nonumber&&+\sum_{k=0}^{m}\sum_{i=0}^{n}\sum_{\sigma\in\mathbb S_{(i,n-i)}}\sum_{ \tau\in\mathbb S_{(k,m-k)}}\varepsilon(\sigma)\varepsilon(\tau)(-1)^{x_{\sigma(1)}+\ldots+x_{\sigma(i)}}\\
				\nonumber&&\frkR_{i,m-k+1}\Big(x_{\sigma(1)},\ldots,x_{\sigma(i)};\frkR_{n-i,k}\big(x_{\sigma(i+1)},\ldots,x_{\sigma(n)};u_{\tau(1)},\ldots,u_{\tau(k)}\big),u_{\tau(k+1)},\ldots,u_{\tau(m)}\Big)\\
				\label{open-closed homotopy}&=&0,
			\end{eqnarray}
		}
		\begin{eqnarray}
			\nonumber&&\sum_{i=1}^{n}\sum_{\sigma\in \mathbb S_{(i,n-i)}}\varepsilon(\sigma)\frkR_{n-i+1,m}\Big(l_i(x_{\sigma(1)},\ldots,x_{\sigma(i)})x_{\sigma(i+1)}\ldots x_{\sigma(n)}\otimes u_1 \ldots u_m\Big)\\
			\nonumber&&+\sum_{i=0}^{n}\sum_{k=0}^{m}\sum_{\sigma\in\mathbb S_{(i,n-i)}}\sum_{ \tau\in\mathbb S_{(k,m-k)}}\varepsilon(\sigma)\varepsilon(\tau)(-1)^{x_{\sigma(1)}+\ldots+x_{\sigma(i)}}\\
			\nonumber&&\frkR_{i,m-k+1}\Big(x_{\sigma(1)} \ldots x_{\sigma(i)}\otimes \frkR_{n-i,k}\big(x_{\sigma(i+1)} \ldots x_{\sigma(n)}\otimes u_{\tau(1)} \ldots u_{\tau(k)}\big) u_{\tau(k+1)} \ldots u_{\tau(m)}\Big)\\
			\label{open-closed homotopy}&=&0,
		\end{eqnarray}
		for $n+m\ge1$ and $ x_1,\ldots,x_n\in\g,u_1,\ldots,u_m\in\h.$
	\end{defi}
	
	\begin{rmk}
		When $n=0$,  \eqref{open-closed homotopy} implies that $\{\frkR_{0,q}\}_{q=1}^{+\infty}$ is an $L_\infty$-algebra structure on the graded vector space $\h$. Moreover, if  $\frkR_{p,0}=0$ for all $p\ge1$, then an open-closed homotopy Lie algebra $(\g,\h,\{l_k\}_{k=1}^{+\infty},\{\frkR_{p,q}\}_{p+q\ge1})$ reduces to an $L_\infty$-action of the $L_\infty$-algebra $(\g,\{l_k\}_{k=1}^{+\infty})$ on the  $L_\infty$-algebra $(\h,\{\frkR_{0,q}\}_{q=1}^{+\infty})$. See \cite[Definition 2.4]{CC} for more details about   $L_\infty$-actions of   $L_\infty$-algebras.
		In general, an open-closed homotopy Lie algebra is equivalent to a (nonabelian) extension of $L_\infty$-algebras.  In particular, there is an $L_\infty$-algebra structure $\{\huaR_k\}_{k=1}^{+\infty}$ on $\g\oplus \h$ given by
		\begin{eqnarray}
			\nonumber&&\huaR_k\big((x_1,u_1),\ldots,(x_k,u_k)\big)\\
			\label{extension-homotopy-lie}&=&\Big(l_k(x_1,\ldots,x_k),\sum_{i=0}^{k}\sum_{\sigma\in \mathbb S_{(i,k-i)} }\varepsilon(\sigma)\frkR_{i,k-i}(x_{\sigma(1)} \ldots x_{\sigma(i)}\otimes u_{\sigma(i+1)} \ldots u_{\sigma(k)})\Big).
		\end{eqnarray}
		One can say that the $L_\infty$-algebra $(\g\oplus \h,\{\huaR_k\}_{k=1}^{+\infty})$ is an extension of $(\g,\{l_k\}_{k=1}^{+\infty})$ by $(\h,\{\frkR_{0,q}\}_{q=1}^{+\infty})$.
	\end{rmk}

Now we give an explicit equivalent description of open-closed homotopy Lie algebras.	Let $(\g,\{l_k\}_{k=1}^{+\infty})$ and  $(\h,\{\mu_k\}_{k=1}^{+\infty})$ be $L_\infty$-algebras, and $\{{\rho_k}\}_{k=1}^{+\infty}$ be a graded linear map of degree $0$ from $\bar{\Sym}(\g)$ to $s^{-1}\coDer(\Sym(\h))$. Define $\frkR_{p,q}:\Sym^p(\g)\otimes \Sym^q(\h)\lon\h$ by
\begin{eqnarray}\label{open-closed-to-lie-homo}
	\frkR_{p,q}(x_1 \ldots x_p\otimes u_1 \ldots u_q):=\left\{
	\begin{array}{ll}
		\mu_q(u_1,\ldots,u_q), &\mbox {$p=0,q\ge 1$, }\\
		\pr_\h\big((s\rho_p)(x_1,\ldots,x_p)(u_1,\ldots,u_q)\big), &\mbox {$p\ge1,q\ge 0$.}
	\end{array}
	\right.
\end{eqnarray}

	\begin{thm}{\rm (\cite{Laz',MZ})}
		With the above notations, $(\g,\h,\{l_k\}_{k=1}^{+\infty},\{\frkR_{p,q}\}_{p+q\ge1})$ is an open-closed homotopy Lie algebra if and only if $\{{\rho_k}\}_{k=1}^{+\infty}$ is an $L_\infty$-algebra homomorphism from $(\g,\{l_k\}_{k=1}^{+\infty})$ to $s^{-1}\coDer(\Sym(\h))$.
	\end{thm}
\begin{rmk}	
An  $L_\infty$-action of an $L_\infty$-algebra $(\g,\{l_k\}_{k=1}^{+\infty})$ on an  $L_\infty$-algebra $(\h,\{\mu_k\}_{k=1}^{+\infty})$ is equivalent to an $L_\infty$-algebra homomorphism $\rho=\{\rho_k\}_{k=1}^{+\infty}$ from $\g$ to $s^{-1}\coDer(\overline{\Sym}(\h))$. Such an action is, therefore, a special case of an open-closed homotopy Lie algebra.
\end{rmk}

	\begin{thm}\label{homotopy-post-lie-to-lie-cor}
		Let $(\g,\{\frkM_{p,q}\}_{p\ge0,q\ge1})$ be a post-Lie$_\infty$ algebra. Define $l^C_k$ by
		\begin{eqnarray}\label{sub-homotopy-lie}
			l^C_k(x_1,\ldots,x_k)\triangleq \sum_{j=0}^{k-1}\sum_{\sigma\in\mathbb S_{(j,k-j)}}\varepsilon(\sigma)\frkM_{j,k-j}(x_{\sigma(1)} \ldots x_{\sigma(j)}\otimes x_{\sigma(j+1)} \ldots x_{\sigma(k)}),
		\end{eqnarray}
		for all $x_1,\ldots,x_k\in \g$. Then $(\g,\{l^C_k\}_{k=1}^{+\infty})$ is an $L_\infty$-algebra, which is denoted by $\g^C$ and called  the {\bf sub-adjacent $L_\infty$-algebra} of the post-Lie$_\infty$ algebra $(\g,\{\frkM_{p,q}\}_{p\ge0,q\ge1})$. 
		
		Moreover, $(\g,\g,\{l^C_k\}_{k=1}^{+\infty},\{\frkM_{p,q}\}_{p+q\ge1})$ is an open-closed homotopy Lie algebra such that  $\frkM_{p,0}=0$ for all $p\ge1$.
	\end{thm}
	
	\begin{proof}
		By Definition \ref{defi:homtopy-post-lie} and Proposition \ref{important-iso},  $\{\frkM_{p,q}\}_{p\ge0,q\ge1}$ is a post-Lie$_\infty$ algebra structure on $\g$ if and only if $m=\sum_{p\ge 0,q\ge 1}\widehat{\frkM}_{p,q}$ is a differential of the cocommutative cotrialgebra $(\Sym(V)\otimes \bar{\Sym}(V),\Delta,\delta)$. By the linear duality, we also  obtain that $\{\frkM_{p,q}\}_{p\ge0,q\ge1}$ is a post-Lie$_\infty$ algebra structure on $\g$ if and only if
		$m^*$ is a continuous differential on the  free complete commutative trialgebra $\widehat{\huaComTrias}(\g^*)$. Note that an $L_\infty$-algebra structure on $\g$ is a continuous differential  \cite{Laz'} on the free complete nonunital commutative algebra $\widehat{\bar{\Sym}}(\g^*)$. We will now show how to construct a continuous differential on $\widehat{\bar{\Sym}}(\g^*)$ from  $m^*$.
		
		Denote the category of graded  vector spaces by $\huaVect$, the category of graded  nonunital commutative algebras by $\huaCom$ and the category of commutative trialgebras by $\huaComTrias$. There is an embedding functor $\frki:\huaCom\lon \huaComTrias$, which is defined by $\frki(A,\cdot):=(A,\cdot,\cdot)$. Let $(A,\cdot,\ast)$ be a graded commutative trialgebra.  We consider the $2$-sided ideal $I$ of the graded commutative trialgebra $(A,\cdot,\ast)$ which is generated by all homogeneous elements  of the form
		$$
		x\cdot y-x\ast y,~\forall x,y\in A.
		$$
		Then the quotient algebra $A/I$ is a graded commutative algebra and we denote it by $A_{\huacom}$. We have constructed a functor $$(-)_{\huacom}:\huaComTrias\lon \huaCom.$$ By a direct calculation we deduce  that $(-)_{\huacom}$ is left adjoint to $\frki$. Since the free  commutative trialgebra functor is left adjoint to the forgetful functor from $\huaComTrias$ to $\huaVect$, we conclude that $\huaComTrias(\g^*)_{\huacom}$ is isomorphic to the free nonunital commutative algebra $\bar{\Sym}(\g^*)$. By completion, we obtain that $\widehat{\huaComTrias}(\g^*)_{\huacom}\cong\widehat{\bar{\Sym}}(\g^*).$
		
Moreover, let $d$ be a continuous derivation of $\widehat{\huaComTrias}(\g^*)$. We have:
		\emptycomment{
\begin{eqnarray*}
			&&d(x\ast y-(-1)^{xy}y\ast x)\\
			&=&d(x)\ast y+(-1)^{xd}x\ast d(y)-(-1)^{xy}\big(d(y)\ast x+(-1)^{yd}y\ast d(x)\big)\\
			&=&d(x)\ast y-(-1)^{y(x+d)}y\ast d(x)+(-1)^{xd}\big(x\ast d(y)-(-1)^{x(y+d)}d(y)\ast x\big),
		\end{eqnarray*}
		and
}
		\begin{eqnarray*}
			&&d(x\cdot y-x\ast y)\\
			&=&d(x)\cdot y+(-1)^{xd}x\cdot d(y)-\big(d(x)\ast y+(-1)^{xd}x\ast d(y)\big)\\
			&=&d(x)\cdot y-d(x)\ast y+(-1)^{xd}\big(x\cdot d(y)-x\ast d(y)\big).
		\end{eqnarray*}
		Therefore, we obtain that $d(I)\subset I$, which implies that $d$ restricts to a unique continuous derivation $\widetilde{d}$ on $\widehat{\bar{\Sym}}(\g^*)$. More precisely, $\widetilde{d}:\widehat{\bar{\Sym}}(\g^*)\lon\widehat{\bar{\Sym}}(\g^*)$ is given by
		\begin{eqnarray}
			\widetilde{d}([x]):=[d(x)],\,\quad\forall d\in \Der(\widehat{\huaComTrias}(\g^*)),~ x\in \widehat{\huaComTrias}(\g^*).
		\end{eqnarray}
	Furthermore, given $v\in \widehat{\Sym}(\g^*)$ and $v'\in \widehat{\bar{\Sym}}(\g^*)$ we have, modulo the relation $*=\cdot$:
	\begin{align*}
	1\otimes vv'=(1\otimes v)\cdot (1\otimes v')=(1\otimes v)* (1\otimes v')=v\otimes v'
	\end{align*}	
and similarly $1\otimes vv'=(-1)^{|v||v'|}1\otimes v'v=(-1)^{|v||v'|}v'\otimes v$. It follows that we have the following isomorphism of commutative trialgebras
\begin{equation}\label{eq:quotient} [\widehat{\Sym}(\g^*)\otimes \widehat{\bar{\Sym}}(\g^*)]/I\cong \prod_{\sigma\in \mathbb S_{(n,m)}\atop n\geq 0, m\geq 1} [\widehat{\Sym}^n(\g^*) \otimes \widehat{\bar{\Sym}^m}(\g^*)]_\sigma\cong \widehat{\bar{\Sym}}(\g^*).
\end{equation}
The structure maps $\{l^C_k\}_{k=1}^{+\infty}$ of the corresponding sub-adjacent $L_\infty$  algebra are dual to the induced differential $\tilde{d}$ on $\widehat{\bar{\Sym}}(\g^*)$.
It follows from (\ref{eq:quotient}) that the dual map
$\Sym(\g)\to\Sym(\g)\otimes\bar{\Sym}(\g)$ has the form
\[
x_1\ldots x_k
\mapsto\sum_{j=0}^{k-1}\sum _{\sigma\in \mathbb S_{(j,k-j)}}
 \epsilon(\sigma)(x_{\sigma(1)}\ldots x_{\sigma(j)})\otimes(x_{\sigma(j+1)}\ldots x_{\sigma(k)})
\]
which implies that the maps $\{l^C_k\}_{k=1}^{+\infty}$ are as stated (the above argument is similar to deducing the structure maps of the $L_\infty$-algebra associated to an $A_\infty$-algebra).

A similar, but simpler, argument, shows that the collection of tensors 		$(\g,\{\frkM_{0,q}\}_{q\ge1})$ forms an $L_\infty$-algebra. Indeed, in the dual language this corresponds to setting $*=0$ in the free commutative trialgebra $\widehat{\Sym}(\g^*)\otimes \widehat{\bar{\Sym}}(\g^*)$ and observing that the resulting structure is a square zero derivation of $\widehat{\bar{\Sym}}(\g^*)$, that is an $L_\infty$-structure on the graded vector space $\g$.

Finally, by \eqref{sub-homotopy-lie}, we have
{\footnotesize{
\begin{eqnarray*}
&&\sum_{i=1}^{n}\sum_{\sigma\in \mathbb S_{(i,n-i)}}\varepsilon(\sigma)\frkM_{n-i+1,m}\Big(l^C_i(x_{\sigma(1)}\ldots x_{\sigma(i)}),x_{\sigma(i+1)},\ldots,x_{\sigma(n)}\otimes x_{n+1}\ldots x_{n+m}\Big)\\
&&+\sum_{i=0}^{n}\sum_{k=1}^{m}\sum_{\sigma\in\mathbb S_{(i,n-i)}}\sum_{ \tau\in\mathbb S_{(k,m-k)}}\varepsilon(\sigma)\varepsilon(\tau)(-1)^{x_{\sigma(1)}+\ldots+x_{\sigma(i)}}\\
&&\frkM_{i,m-k+1}\Big(x_{\sigma(1)}\ldots x_{\sigma(i)}\otimes \frkM_{n-i,k}\big(x_{\sigma(i+1)}\ldots x_{\sigma(n)}\otimes x_{n+\tau(1)}\ldots x_{n+\tau(k)}\big),x_{n+\tau(k+1)}\ldots x_{n+\tau(m)}\Big)\\
&\stackrel{\eqref{sub-homotopy-lie}}{=}&\sum_{i=1}^{n}\sum_{\sigma\in \mathbb S_{(i,n-i)}}\sum_{j=0}^{i-1}\sum_{\tau\in\mathbb S_{(j,i-j)}}\varepsilon(\sigma)\varepsilon(\tau)\\
&&\frkM_{n-i+1,m}\Big(\frkM_{j,i-j}(x_{\sigma(\tau(1))} \ldots x_{\sigma(\tau(j))}\otimes x_{\sigma(\tau(j+1))} \ldots x_{\sigma(\tau(k))}),x_{\sigma(i+1)}\ldots x_{\sigma(n)}\otimes x_{n+1}\ldots x_{n+m}\Big)\\
&&+\sum_{i=0}^{n}\sum_{k=1}^{m}\sum_{\sigma\in\mathbb S_{(i,n-i)}}\sum_{ \tau\in\mathbb S_{(k,m-k)}}\varepsilon(\sigma)\varepsilon(\tau)(-1)^{x_{\sigma(1)}+\ldots+x_{\sigma(i)}}\\
&&\frkM_{i,m-k+1}\Big(x_{\sigma(1)}\ldots x_{\sigma(i)}\otimes \frkM_{n-i,k}\big(x_{\sigma(i+1)}\ldots x_{\sigma(n)}\otimes x_{n+\tau(1)}\ldots x_{n+\tau(k)}\big),x_{n+\tau(k+1)},\ldots,x_{n+\tau(m)}\Big)\\
&=&\sum_{\sigma\in \mathbb S_{(j,i,n-j-i)}\atop 0\le j\le n-1, 1\le i\le n}\varepsilon(\sigma)\frkM_{n-j-i+1,m}\Big(\frkM_{j,i}(x_{\sigma(1)}\ldots x_{\sigma(j)}\otimes x_{\sigma(j+1)}\ldots x_{\sigma(j+i)})x_{\sigma(j+i+1)}\ldots x_{\sigma(n)}\otimes x_{n+1}\ldots x_{n+m}\Big)\\
					\nonumber&&+\sum_{\sigma\in\mathbb S_{(n-j,j)}\atop 0\le j\le n}\sum_{ \tau\in\mathbb S_{(i,m-i)}\atop 1\le i\le m}\varepsilon(\sigma)\varepsilon(\tau)(-1)^{x_{\sigma(1)}+\ldots+x_{\sigma(n-j)}}\\
					\nonumber&&\frkM_{n-j,m-i+1}\Big(x_{\sigma(1)} \ldots x_{\sigma(n-j)}\otimes \frkM_{j,i}\big(x_{\sigma(n-j+1)} \ldots x_{\sigma(n)}\otimes x_{n+\tau(1)} \ldots x_{n+\tau(i)}\big) x_{n+\tau(i+1)} \ldots x_{n+\tau(m)}\Big)\\
&\stackrel{\eqref{homotopy-post-lie}}{=}&0.
\end{eqnarray*}
}}
Thus, $(\g,\g,\{l^C_k\}_{k=1}^{+\infty},\{\frkM_{p,q}\}_{p+q\ge1})$ is an open-closed homotopy Lie algebra such that  $\frkM_{p,0}=0$ for all $p\ge1$.
	\end{proof}

\begin{thm}\label{th:characterization-more}
Let $\g$ be a graded vector space and $\frkM_{p,q}:\Sym^p(\g)\otimes \Sym^q(\g)\lon\g,~p\ge 0,q\ge1$  be a family of graded linear maps of degree $1$.
Then $(\g,\{\frkM_{p,q}\}_{p\ge0,q\ge1})$ is a post-Lie$_\infty$ algebra if and only if
 $(\g,\g,\{l^C_k\}_{k=1}^{+\infty},\{\frkM_{p,q}\}_{p+q\ge1})$ is an open-closed homotopy Lie algebra which satisfies
\begin{eqnarray*}
l^C_k(x_1,\ldots,x_k)&=&\sum_{j=0}^{k-1}\sum_{\sigma\in\mathbb S_{(j,k-j)}}\varepsilon(\sigma)\frkM_{j,k-j}(x_{\sigma(1)} \ldots x_{\sigma(j)}\otimes x_{\sigma(j+1)} \ldots x_{\sigma(k)}),\\
\frkM_{p,0}&=&0,\,\,\,\forall p\ge 1.
		\end{eqnarray*}
	\end{thm}
\begin{proof}
By Theorem \ref{homotopy-post-lie-to-lie-cor}, we obtain that $(\g,\g,\{l^C_k\}_{k=1}^{+\infty},\{\frkM_{p,q}\}_{p+q\ge1})$ is an open-closed homotopy Lie algebra satisfying $\frkM_{p,0}=0$ for all $p\ge1$.

Conversely, by  the proof of Theorem \ref{homotopy-post-lie-to-lie-cor}, we deduce that the equation \eqref{open-closed homotopy} is precisely  the equation \eqref{homotopy-post-lie}. Then we obtain that $(\g,\{\frkM_{p,q}\}_{p\ge0,q\ge1})$ is a post-Lie$_\infty$ algebra.
\end{proof}
	\subsection{Post-Lie$_\infty$ algebras and higher geometric structures}
	
	\begin{defi}\label{L-infty-action-on-dgca}{\rm (\cite{Hueb})}
		Let $(\g,\{l_k\}_{k=1}^{+\infty})$ be an $L_\infty$-algebra and $(A,\cdot,d)$  be a dg commutative algebra. An {\bf action} of the $L_\infty$-algebra~$(\g,\{l_k\}_{k=1}^{+\infty})$ on the dg commutative algebra~$(A,\cdot,d)$ is a family of  homomorphisms $\{\rho_k:\Sym^k(\g)\lon \Der(A)\}_{k=1}^{+\infty}$ of graded vector spaces  of degree $1$
		for any $k\geq 1$ satisfying, for every collection of
		homogeneous elements $x_1,\ldots,x_n\in \g$,
		\begin{eqnarray}
			\nonumber&&\sum_{i=1}^n\sum_{\sigma\in \mathbb S_{(i,n-i)} }\varepsilon(\sigma)\rho_{n-i+1}\Big(l_i(x_{\sigma(1)},\ldots,x_{\sigma(i)}),x_{\sigma(i+1)},\ldots,x_{\sigma(n)}\Big)\\
			\label{L-infty-homo-on-dgca}&&+\sum_{i=1}^{n-1}\sum_{\sigma\in \mathbb S_{(i,n-i)}}(-1)^{x_{\sigma(1)}+\cdots+x_{\sigma(i)}}\rho_i(x_{\sigma(1)},\ldots,x_{\sigma(i)})\circ \rho_{n-i}(v_{\sigma(k_1+\ldots+k_{i-1}+1)},\ldots,v_{\sigma(n)})\Big)\\
			\nonumber&&+d\circ \rho_n(x_1,\ldots,x_n)-(-1)^{x_1+\cdots+x_n+1}\rho_n(x_1,\ldots,x_n)\circ d\\
			\nonumber&=&0.
		\end{eqnarray}
		Equivalently, $\{s^{-1}\circ\rho_k\}_{k=1}^{+\infty}$ is an $L_\infty$-algebra homomorphism form  $(\g,\{l_k\}_{k=1}^{+\infty})$ to $s^{-1}\Der(A)$.
	\end{defi}

	\begin{thm}\label{action-L-infty}
		Let $\{\rho_k:\Sym^k(\g)\lon \Der(A)\}_{k=1}^{+\infty}$ be an action of an $L_\infty$-algebra $(\g,\{l_k\}_{k=1}^{+\infty})$ on a dg commutative algebra $(A,\cdot,d)$. Then there is an $L_\infty$-algebra structure on $A\otimes \g$, which is defined by
		\begin{eqnarray*}
			\frkl_1(a_1\otimes x_1)&=&d(a_1)\otimes x_1+(-1)^{a_1}a_1\otimes l_1(x_1),\\
			\frkl_k(a_1\otimes x_1,\ldots,a_k\otimes x_k)&=&(-1)^{\sum_{i=1}^{k}(1+x_1+\ldots+x_{i-1})a_i}a_1\ldots a_k\otimes l_k(x_1,\ldots,x_k)\\
			\nonumber&&+\sum_{j=1}^{k}(-1)^{(1+x_1+\ldots+x_{j-1}+a_{j+1}+x_{j+1}+\ldots+a_{k}+x_{k})a_j+(x_{j+1}+\ldots+x_k)x_j}\\
			&&(-1)^{\sum_{i=1}^{k}(1+x_1+\ldots+x_{i-1})a_i}a_1\ldots \hat{a}_j\ldots a_k\rho_{k-1}(x_1,\ldots,\hat{x}_j,\ldots,x_k)(a_{j})\otimes x_{j}.
		\end{eqnarray*}
		It is called the action $L_\infty$-algebra of $\{\rho_k:\Sym^k(\g)\lon \Der(A)\}_{k=1}^{+\infty}$.
	\end{thm}
For the proof (though not for a formulation) we will need the notion of an {\bf $L_\infty$-algebroid}, also known as a homotopy Lie-Rinehart algebra. Recall  from \cite{Hueb,Yu}
that an  $L_\infty$-algebroid over a graded  commutative algebra ${A}$ is an $L_\infty$-algebra $(\g,\{l_k\}_{k=1}^{+\infty})$ together with an ${A}$-module structure on $\g$ and an  action $\{\rho_k:\Sym^k(\g)\lon \Der(A)\}_{k=1}^{+\infty}$ (called the anchor)  such that
\begin{eqnarray*}
	\rho(ax_1,\ldots,x_k)&=&(-1)^{a}a\rho(x_1,\ldots,x_k),\\
	l_{k+1}(x_1,\ldots,x_k,ax_{k+1})&=&\rho_k(x_1,\ldots,x_k)(a)x_{k+1}+(-1)^{a(1+x_1+\ldots+x_k)}al_{k+1}(x_1,\ldots,x_{k+1}),
\end{eqnarray*}
here $x_1,\ldots,x_{k+1}\in \g,~a\in {A}.$ We usually denote an $L_\infty$-algebroid over ${A}$ by $({A},\g,\{l_k\}_{k=1}^{+\infty},\{\rho_k\}_{k=1}^{+\infty})$.
	\begin{proof}[Proof of Theorem \ref{action-L-infty}]
	The given an action $\{\rho_k:\Sym^k(\g)\lon \Der(A)\}_{k=1}^{+\infty}$ gives rise to the structure of an $L_\infty$-algebroid on the pair $(A,A\otimes \g)$, cf. \cite[Corollary 38]{Vita} where this statement is proved in a more general situation when $\g$ is itself an $L_\infty$-algebroid acting on $A$. A straightforward computation shows that the structure maps of this $L_\infty$-algebroid are as stated. Since an $L_\infty$-algebra is an $L_\infty$-algebroid of a special kind, the desired statement follows.
	\end{proof}
	
	Let $\{\rho_k:\Sym^k(\g)\lon \Der(A)\}_{k=1}^{+\infty}$ be an action as above. We define the graded linear map $$\frkM_{p,q}:\Sym^p(A\otimes \g)\otimes \Sym^q(A\otimes \g)\lon A\otimes \g$$ by
	\emptycomment{
		{\small{
				\begin{eqnarray*}\label{action-homo-post}
					\frkM_{p,q}\big((a_1\otimes x_1\ldots a_p\otimes x_p)\otimes(a_{p+1}\otimes x_{p+1}\ldots a_{p+q}\otimes x_{p+q})\big)=\left\{
					\begin{array}{lcll}
						d(a_1)\otimes x_1+(-1)^{a_1}a_1\otimes l_1(x_1)             &p=0,~q=1\\
						(-1)^{}a_1\ldots a_q\otimes l_q(x_1,\cdots,x_q)                 &p=0,~q\ge 2\\
						\dM_{\rho}^{(1,0)}X_0                             &p\ge 1,~q=1\\
						0      &\text{for other cases}
					\end{array}
					\right.
				\end{eqnarray*}
		}}
	}
	\begin{eqnarray*}
		\frkM_{0,1}(a_1\otimes x_1)&=&d(a_1)\otimes x_1+(-1)^{a_1}a_1\otimes l_1(x_1)\\
		\frkM_{0,q}(a_1\otimes x_1\ldots a_q\otimes x_q)&=&(-1)^{\sum_{i=1}^{q}(1+x_1+\ldots+x_{i-1})a_i}a_1\ldots a_q\otimes l_q(x_1,\ldots,x_q)\\
		\frkM_{p,1}\big((a_1\otimes x_1\ldots a_p\otimes x_p)\otimes(a_{p+1}\otimes x_{p+1})\big)&=&(-1)^{\sum_{i=1}^{p}(1+x_1+\ldots+x_{i-1})a_i}a_1\ldots a_p\rho_{p}(x_1,\ldots,x_p)(a_{p+1})\otimes x_{p+1},
	\end{eqnarray*}
	and
	$\frkM_{p,q}=0$ for other cases.

	\begin{thm}\label{homotopy-action-to-homo-post}
		With the notation as above, $(A\otimes \g,\{\frkM_{p,q}\}_{p\ge0,q\ge1})$ is a post-Lie$_\infty$ algebra, whose sub-adjacent $L_\infty$-algebra is exactly the  action $L_\infty$-algebra.
	\end{thm}
	\begin{proof}
		By \eqref{sub-homotopy-lie}, we note that $(A\otimes \g,\{l^C_k\}_{k=1}^{+\infty})$ is the action $L_\infty$-algebra. Since $\{\rho_k:\Sym^k(\g)\lon \Der(A)\}_{k=1}^{+\infty}$ is an action of the $L_\infty$-algebra $(\g,\{l_k\}_{k=1}^{+\infty})$ on the dg commutative algebra $(A,\cdot,d)$, we deduce that $\{\frkM_{p,1}\}_{p=1}^{+\infty}$ is an action of the action $L_\infty$-algebra $(A\otimes \g,\{l^C_k\}_{k=1}^{+\infty})$ on the  $L_\infty$-algebra $(A\otimes \g,\{\frkM_{0,q}\}_{q=1}^{+\infty})$. Then, by Theorem \ref{th:characterization-more},   $(A\otimes \g,\{\frkM_{p,q}\}_{p\ge0,q\ge1})$ is a post-Lie$_\infty$ algebra.
	\end{proof}

	Recall from \cite{CF,LSX,OP,Vo}
	that a {\bf homotopy Poisson algebra} is a graded commutative algebra $(A,\cdot)$ with an $L_\infty$-algebra structure $\{l_k\}_{k=1}^{+\infty}$, such that
	\begin{eqnarray*}
		l_{k}(x_1,\ldots,x_{k-1},x\cdot y)=l_{k}(x_1,\ldots,x_{k-1},x)\cdot y+(-1)^{x(1+x_1+\ldots+x_{k-1})}x\cdot l_{k}(x_1,\ldots,x_{k-1},y),
	\end{eqnarray*}
	here $x_1,\ldots,x_{k-1},x,y\in {A}.$ We usually denote a homotopy Poisson algebra by $({A},\cdot,\{l_k\}_{k=1}^{+\infty})$.

	\begin{cor}\label{homotopy-poisson-to-homo-post}
		Let $({A},\cdot,\{l_k\}_{k=1}^{+\infty})$ be a homotopy Poisson algebra. Then  there is a post-Lie$_\infty$ algebra structure on $A\otimes A$  given by
		\begin{eqnarray*}
			\frkM_{0,1}(a_1\otimes x_1)&=&l_1(a_1)\otimes x_1+(-1)^{a_1}a_1\otimes l_1(x_1)\\
			\frkM_{0,q}(a_1\otimes x_1\ldots a_q\otimes x_q)&=&(-1)^{\sum_{i=1}^{q}(1+x_1+\ldots+x_{i-1})a_i}a_1\ldots a_q\otimes l_q(x_1,\ldots,x_q),\\
			\frkM_{p,1}\big((a_1\otimes x_1\ldots a_p\otimes x_p)\otimes(a_{p+1}\otimes x_{p+1})\big)&=&(-1)^{\sum_{i=1}^{p}(1+x_1+\ldots+x_{i-1})a_i}a_1\ldots a_pl_{p+1}(x_1,\ldots,x_p,a_{p+1})\otimes x_{p+1},
		\end{eqnarray*}
		and
		$\frkM_{p,q}=0$ for other cases.
	\end{cor}
	
	\begin{proof}
		Since $\{\ad_k:\Sym^k(A)\lon \Der(A)\}_{k=1}^{+\infty}$ is an  action of the $L_\infty$-algebra~$(A,\{l_k\}_{k=1}^{+\infty})$ on the dg commutative algebra~$(A,\cdot,l_1)$, by Theorem \ref{homotopy-action-to-homo-post}, we deduce that $(A\otimes A,\{\frkM_{p,q}\}_{p\ge0,q\ge1})$ is a post-Lie$_\infty$ algebra.
	\end{proof}

	\begin{cor}\label{homotopy-lie-algeoid-to-homo-post}
		Let $({A},\g,\{l_k\}_{k=1}^{+\infty},\{\rho_k\}_{k=1}^{+\infty})$ be an $L_\infty$-algebroid. Then  there is a post-Lie$_\infty$ algebra structure on $A\otimes \g$  given by
		\begin{eqnarray*}
			\frkM_{0,q}(a_1\otimes x_1\ldots a_q\otimes x_q)&=&(-1)^{\sum_{i=1}^{q}(1+x_1+\ldots+x_{i-1})a_i}a_1\ldots a_q\otimes l_q(x_1,\ldots,x_q),\\
			\frkM_{p,1}\big((a_1\otimes x_1\ldots a_p\otimes x_p)\otimes(a_{p+1}\otimes x_{p+1})\big)&=&(-1)^{\sum_{i=1}^{p}(1+x_1+\ldots+x_{i-1})a_i}a_1\ldots a_p\rho_{p}(x_1,\ldots,x_p)(a_{p+1})\otimes x_{p+1},
		\end{eqnarray*}
		and
		$\frkM_{p,q}=0$ for other cases.
	\end{cor}
	\begin{proof}
	Again, this is a straightforward consequence of Theorem \ref{homotopy-action-to-homo-post}.
	\end{proof}	
	
		To sum up, Theorem \ref{homotopy-action-to-homo-post}, Corollary \ref{homotopy-poisson-to-homo-post} and Corollary \ref{homotopy-lie-algeoid-to-homo-post} demonstrate that there is  a close relationship between post-Lie$_\infty$ algebras and higher geometric structures.


	\section{Homotopy Rota-Baxter operators and post-Lie$_\infty$ algebras}
	\label{sec:rb}
	
	In this section, we introduce the notion of homotopy Rota-Baxter operators on  open-closed homotopy Lie algebras and establish their relationships with post-Lie$_\infty$ algebras.

	\subsection{Homotopy Rota-Baxter operators on open-closed homotopy Lie algebras}
	
	We first recall higher derived brackets, by which we construct an $L_\infty$-algebra from an open-closed homotopy Lie algebra. Then homotopy Rota-Baxter operators on an open-closed homotopy Lie algebra are defined to be Maurer-Cartan elements of this $L_\infty$-algebra.

	\begin{defi}{\rm (\cite{Vo})}
		A $V$-data consists of a quadruple $(L,H,P,\Phi)$ where
		\begin{itemize}
			\item[$\bullet$] $(L,[\cdot,\cdot])$ is a graded Lie algebra,
			\item[$\bullet$] $H$ is an abelian graded Lie subalgebra of $(L,[\cdot,\cdot])$,
			\item[$\bullet$] $P:L\lon L$ is a projection, that is $P\circ P=P$, whose image is $H$ and kernel is a  graded Lie subalgebra of $(L,[\cdot,\cdot])$,
			\item[$\bullet$] $\Phi$ is an element of $\ker(P)^1$ such that $[\Phi,\Phi]=0$.
		\end{itemize}
	\end{defi}
	
	\begin{thm}{\rm (\cite{Vo})}\label{thm:db}
		Let $(L,H,P,\Phi)$ be a $V$-data. Then $(H,\{{l_k}\}_{k=1}^{+\infty})$ is an $L_\infty$-algebra where
		\begin{eqnarray}\label{V-shla}
			l_k(a_1,\ldots,a_k)=P[\ldots[[\Phi,a_1],a_2],\ldots,a_k],
		\end{eqnarray}
		for homogeneous elements  $a_1,\ldots,a_k\in H$. We call $\{{l_k}\}_{k=1}^{+\infty}$ the {\bf higher derived brackets} of the $V$-data $(L,H,P,\Phi)$. 
	\end{thm}

	Let $(\g,\h,\{l_k\}_{k=1}^{+\infty},\{\frkR_{p,q}\}_{p+q\ge1})$ be an open-closed homotopy Lie algebra. Any $f\in\Hom(\bar{\Sym}(\h),\g)\subset\Hom(\bar{\Sym}(\g\oplus \h),\g\oplus \h)$ gives rise to a coderivation on $\bar{\Sym}(\g\oplus \h)$, which is denoted by $\hat{f}$.
	\emptycomment{\yh{I do not think the following formula is clear.}
		\begin{eqnarray}\label{extension-coderivation}
			\hat{f}\big(\huaX\odot \huaU\big)=\sum_{\huaU}(-1)^{f\huaX}\huaX\odot\big(f(\huaU_{(1)})\odot\huaU_{(2)}+f(\huaU)\big),\quad \huaX\in{\Sym}(\g),\huaU\in\bar{\Sym}(\h).
		\end{eqnarray}
		Here $\bar{\Delta}(\huaU)=\sum_{\huaU}\huaU_{(1)}\otimes \huaU_{(2)}$
		and $\hat{f}\big(\huaX\big)=0,\,\,\huaX\in\bar{\Sym}(\g).$} Moreover, we denote by $\huaH$ the graded vector space of coderivations on $\bar{\Sym}(\g\oplus \h)$ given by $\Hom(\bar{\Sym}(\h),\g)$.

	\begin{pro}
		Let $(\g,\h,\{l_k\}_{k=1}^{+\infty},\{\frkR_{p,q}\}_{p+q\ge1})$ be an open-closed homotopy Lie algebra. Then the following quadruple is a V-data:
		\begin{itemize}
			\item[$\bullet$] the graded Lie algebra $(L,[\cdot,\cdot])$ is given by $(\coDer(\bar{\Sym}(\g\oplus\h)),[\cdot,\cdot]_C)$;
			\item[$\bullet$] the abelian graded Lie subalgebra $H$ is given by $\huaH;$
			\item[$\bullet$] $P:L\lon L$ is the projection onto the subspace $\huaH$;
			\item[$\bullet$] $\Phi=\hat{\huaR},$ where $\huaR=\sum_{k=1}^{+\infty}\huaR_k$ is given by \eqref{extension-homotopy-lie}.
		\end{itemize}
		Consequently, $(\huaH,\{\frkl_k\}_{k=1}^{+\infty})$ is an $L_\infty$-algebra, where $\frkl_k$ is given by
		\eqref{V-shla}.
	\end{pro}
	
	\begin{proof}
		By \cite[Proposition 3.13 and Proposition 3.14]{CC},  $\huaH$ is an abelian graded Lie subalgebra and $\ker P$ is a graded Lie subalgebra. Since $\huaR=\sum_{k=1}^{+\infty}\huaR_k$ is an $L_\infty$-algebra structure on $\g\oplus\h$, it follows that $[\Phi,\Phi]_C=0$. Moreover,  \eqref{extension-homotopy-lie}, implies $P(\Phi)=0$. Thus $(L,H,P,\Phi)$ is a V-data. Hence by Theorem \ref{thm:db}, we obtain higher derived brackets $\{{\frkl_k}\}_{k=1}^{+\infty}$ on $\huaH$.
	\end{proof}
	
	\begin{defi}\label{homotopy-embedding-tensors-homotopy-lie}
		A degree $0$ element  $\Theta=\sum_{k=1}^{+\infty}\Theta_k\in \Hom(\bar{\Sym}(\h),\g)$ is called a {\bf homotopy   Rota-Baxter operator}  on an open-closed homotopy Lie algebra $(\g,\h,\{l_k\}_{k=1}^{+\infty},\{\frkR_{p,q}\}_{p+q\ge1})$ if it is a Maurer-Cartan element of the  $L_\infty$-algebra $(\huaH,\{{\frkl_k}\}_{k=1}^{+\infty})$; in other words the following identity holds:
		\begin{eqnarray}\label{homotopy-rota-baxter}
			P\big(e^{[\cdot,\hat{\Theta}]_C}\Phi\big)=0.
		\end{eqnarray}
	\end{defi}

	\begin{rmk}
		Homotopy relative Rota-Baxter operators (also called $\huaO$-operators) on $L_\infty$-algebras with respect to $L_\infty$-actions of $L_\infty$-algebras were studied by Caseiro and Nunes da Costa in \cite{CC}. Since   $L_\infty$-actions of $L_\infty$-algebras are special cases of  open-closed homotopy Lie algebras, it follows that   homotopy   Rota-Baxter operators  on  open-closed homotopy Lie algebras contain homotopy relative Rota-Baxter operators  as special cases.
	\end{rmk}
	
\begin{thm}\label{twist-homotopy-lie}
		Let $\Theta=\sum_{k=1}^{+\infty}\Theta_k\in \Hom(\bar{\Sym}(\h),\g)$ be a homotopy   Rota-Baxter operator on an open-closed homotopy Lie algebra $(\g,\h,\{l_k\}_{k=1}^{+\infty},\{\frkR_{p,q}\}_{p+q\ge1})$. Define $\alpha_n:\Sym^n(\h)\lon\h$ by
		\begin{eqnarray}
			\label{descendant-homotopy-lie}&&\alpha_n(u_1,\ldots,u_n)=\frkR_{0,n}(u_1 \ldots  u_n)+\sum_{k=1}^{n}\sum_{i_1+\ldots+i_{k+1}=n\atop i_1,\ldots,i_k\ge 1,i_{k+1}\ge 0}\sum_{\sigma\in\mathbb S_{(i_1,\ldots,i_{k+1})}}\frac{\varepsilon(\sigma)}{k!}\\
			\nonumber&&\frkR_{k,i_{k+1}}\Big(\Theta_{i_1}(u_{\sigma(1)} \ldots  u_{\sigma(i_1)}) \ldots \Theta_{i_{k}}(u_{\sigma(i_1+\ldots+i_{k-1}+1)} \ldots  u_{\sigma(i_1+\ldots+i_{k})})\otimes u_{\sigma(i_1+\ldots+i_{k}+1)} \ldots  u_{\sigma(n)}\Big).
		\end{eqnarray}
		Then $(\h,\{\alpha_k\}_{k=1}^{+\infty})$ is an $L_\infty$-algebra and $\Theta$ is a homomorphism from $(\h,\{\alpha_k\}_{k=1}^{+\infty})$ to $(\g,\{l_k\}_{k=1}^{+\infty})$.
	\end{thm}
	\begin{proof}
		Since $[\cdot,\hat{\Theta}]_C$ is a locally nilpotent derivation on the graded Lie algebra $(\coDer(\bar{\Sym}(\g\oplus\h)),[\cdot,\cdot]_C)$ \cite{LST}, it follows  that $e^{[\cdot,\hat{\Theta}]_C}$ is an automorphism of $(\coDer(\bar{\Sym}(\g\oplus\h)),[\cdot,\cdot]_C)$. Therefore, $e^{[\cdot,\hat{\Theta}]_C}\Phi$ is a differential on the coalgebra $\bar{\Sym}(\g\oplus\h)$, and gives rise to an $L_\infty$-algebra structure on  $\g\oplus\h$.
		
		Since $\Theta\in \Hom(\bar{\Sym}(\h),\g)$, it follows that $\hat{\Theta}$ is a locally nilpotent coderivation of the coalgebra $\bar{\Sym}(\g\oplus\h)$,  and  $e^{\hat{\Theta}}$   is a well-defined automorphism of the  coalgebra $\bar{\Sym}(\g\oplus\h)$. For all $u_1,\ldots,u_n\in\h,$ we have
		\begin{eqnarray}
			\label{iso-coalgebra} &&e^{\hat{\Theta}}(u_1 \ldots  u_n)=u_1 \ldots  u_n+\sum_{k=1}^{n-1}\sum_{i_1+\ldots+i_{k+1}=n\atop i_1,\ldots,i_{k+1}\ge 1}\sum_{\sigma\in\mathbb S_{(i_1,\ldots,i_{k+1})}}\frac{\varepsilon(\sigma)}{k!}\\
			\nonumber &&\Theta_{i_1}(u_{\sigma(1)} \ldots  u_{\sigma(i_1)}) \ldots \Theta_{i_{k}}(u_{\sigma(i_1+\ldots+i_{k-1}+1)} \ldots  u_{\sigma(i_1+\ldots+i_{k})})  u_{\sigma(i_1+\ldots+i_{k}+1)} \ldots  u_{\sigma(n)}\\
			\nonumber &&+\sum_{k=1}^{n}\sum_{i_1+\ldots+i_k=n\atop i_1,\ldots,i_k\ge 1}\sum_{\sigma\in\mathbb S_{(i_1,\ldots,i_k)}}\frac{\varepsilon(\sigma)}{k!}\Theta_{i_1}(u_{\sigma(1)} \ldots  u_{\sigma(i_1)}) \ldots  \Theta_{i_k}(u_{\sigma(i_1+\ldots+i_{k-1}+1)} \ldots  u_{\sigma(n)}).
		\end{eqnarray}
		Moreover, since $[\Phi,\hat{\Theta}]_C=\Phi\circ \hat{\Theta}-\hat{\Theta}\circ \Phi$, and by induction, we have
		\begin{eqnarray}
			\underbrace{[\ldots[[}_k\Phi,\hat{\Theta}]_C,\hat{\Theta}]_C,\ldots,\hat{\Theta}]_C=\sum_{i=0}^{k}(-1)^{i}{k\choose i}\underbrace{\hat{\Theta}\circ\ldots\circ\hat{\Theta}}_i\circ \Phi\circ \underbrace{\hat{\Theta}\circ\ldots\circ\hat{\Theta}}_{k-i}.
		\end{eqnarray}
		Therefore, we have
		\begin{eqnarray}
			\nonumber e^{[\cdot,\hat{\Theta}]_C}\Phi&=&\sum_{k=0}^{+\infty}\frac{1}{k!}\underbrace{[\ldots[[}_k\Phi,\hat{\Theta}]_C,\hat{\Theta}]_C,\ldots,\hat{\Theta}]_C,\\
			\nonumber &=&\sum_{k=0}^{+\infty}\sum_{i=0}^{k}(-1)^{i}\frac{1}{i!(k-i)!}\underbrace{\hat{\Theta}\circ\ldots\circ\hat{\Theta}}_i\circ \Phi\circ \underbrace{\hat{\Theta}\circ\ldots\circ\hat{\Theta}}_{k-i}\\
			\label{twistor}&=&e^{-\hat{\Theta}}\circ \Phi\circ e^{\hat{\Theta}}.
		\end{eqnarray}
		Equivalently, we have the following  commutative diagram:
		\begin{equation}\label{eq:relation}
			\small{
				\xymatrix{\bar{\Sym}(\g\oplus\h) \ar[rr]^{e^{[\cdot,\hat{\Theta}]_C}\Phi}\ar[d]_{e^{\hat{\Theta}}} &  &  \bar{\Sym}(\g\oplus\h) \ar[d]^{e^{\hat{\Theta}}}  \\
					\bar{\Sym}(\g\oplus\h) \ar[rr]^{\Phi} & &\bar{\Sym}(\g\oplus\h),}
			}
		\end{equation}
		which implies that $e^{\hat{\Theta}}$ is a homomorphism from the $L_\infty$-algebra $\big(\bar{\Sym}(\g\oplus\h),\bar{\Delta},e^{[\cdot,\hat{\Theta}]_C}\Phi\big)$ to the $L_\infty$-algebra $\big(\bar{\Sym}(\g\oplus\h),\bar{\Delta},\Phi\big)$. 
		
		By \eqref{homotopy-rota-baxter}, we obtain that $\bar{\Sym}(\h)$ is an $L_\infty$-subalgebra of $\big(\bar{\Sym}(\g\oplus\h),\bar{\Delta},e^{[\cdot,\hat{\Theta}]_C}\Phi\big)$. More precisely, for all $u_1,\ldots,u_n\in\h,$ we have
\emptycomment{
		\begin{eqnarray*}
			&&\pr_{\h}\big(e^{[\cdot,\hat{\Theta}]_C}\Phi\big)(u_1,\ldots, u_n)\\
			&\stackrel{\eqref{twistor}}{=}&\pr_{\h}\big(e^{-\hat{\Theta}}\circ \Phi\circ e^{\hat{\Theta}}\big)(u_1,\ldots, u_n)\\
			&\stackrel{\eqref{iso-coalgebra}}{=}&\sum_{k=1}^{n-1}\sum_{i_1+\ldots+i_{k+1}=n\atop i_1,\ldots,i_{k+1}\ge 1}\sum_{\sigma\in\mathbb S_{(i_1,\ldots,i_{k+1})}}\frac{\varepsilon(\sigma)}{k!}\\
			&&\frkR_{k,i_{k+1}}\Big(\Theta_{i_1}(u_{\sigma(1)} \ldots  u_{\sigma(i_1)}) \ldots \Theta_{i_{k}}(u_{\sigma(i_1+\ldots+i_{k-1}+1)} \ldots  u_{\sigma(i_1+\ldots+i_{k})}) \otimes u_{\sigma(i_1+\ldots+i_{k}+1)} \ldots  u_{\sigma(n)}\Big)\\
			&&+\sum_{k=1}^{n}\sum_{i_1+\ldots+i_k=n\atop i_1,\ldots,i_k\ge 1}\sum_{\sigma\in\mathbb S_{(i_1,\ldots,i_k)}}\frac{\varepsilon(\sigma)}{k!}\\
			&&\frkR_{k,0}\Big(\Theta_{i_1}(u_{\sigma(1)} \ldots u_{\sigma(i_1)}) \ldots  \Theta_{i_k}(u_{\sigma(i_1+\ldots+i_{k-1}+1)} \ldots  u_{\sigma(n)})\Big)\\
			&&+\frkR_{0,n}(u_1 \ldots  u_n)\\
			&=&\sum_{k=1}^{n}\sum_{i_1+\ldots+i_{k+1}=n\atop i_1,\ldots,i_k\ge 1,i_{k+1}\ge 0}\sum_{\sigma\in\mathbb S_{(i_1,\ldots,i_{k+1})}}\frac{\varepsilon(\sigma)}{k!}\\
			&&\frkR_{k,i_{k+1}}\Big(\Theta_{i_1}(u_{\sigma(1)} \ldots  u_{\sigma(i_1)}) \ldots \Theta_{i_{k}}(u_{\sigma(i_1+\ldots+i_{k-1}+1)} \ldots  u_{\sigma(i_1+\ldots+i_{k})})\otimes  u_{\sigma(i_1+\ldots+i_{k}+1)} \ldots  u_{\sigma(n)}\Big)\\
			&&+\frkR_{0,n}(u_1 \ldots  u_n),
		\end{eqnarray*}
lllllllllllll
}
\begin{eqnarray*}
			&&\pr_{\h}\big(e^{[\cdot,\hat{\Theta}]_C}\Phi\big)(u_1,\ldots, u_n)\stackrel{\eqref{twistor}}{=}\pr_{\h}\big(e^{-\hat{\Theta}}\circ \Phi\circ e^{\hat{\Theta}}\big)(u_1,\ldots, u_n)\\
			&\stackrel{\eqref{iso-coalgebra}}{=}&\sum_{k=1}^{n}\sum_{i_1+\ldots+i_{k+1}=n\atop i_1,\ldots,i_k\ge 1,i_{k+1}\ge 0}\sum_{\sigma\in\mathbb S_{(i_1,\ldots,i_{k+1})}}\frac{\varepsilon(\sigma)}{k!}\\
			&&\frkR_{k,i_{k+1}}\Big(\Theta_{i_1}(u_{\sigma(1)} \ldots  u_{\sigma(i_1)}) \ldots \Theta_{i_{k}}(u_{\sigma(i_1+\ldots+i_{k-1}+1)} \ldots  u_{\sigma(i_1+\ldots+i_{k})})\otimes  u_{\sigma(i_1+\ldots+i_{k}+1)} \ldots  u_{\sigma(n)}\Big)\\
			&&+\frkR_{0,n}(u_1 \ldots  u_n),
		\end{eqnarray*}
		which is exactly $\alpha_n$ given by \eqref{descendant-homotopy-lie}.
		
		Since $(\g,\h,\{l_k\}_{k=1}^{+\infty},\{\frkR_{p,q}\}_{p+q=1}^{+\infty})$ is an open-closed homotopy Lie algebra, we obtain the short exact sequence of $L_\infty$-algebras with strict morphisms:
		$$ 0\longrightarrow\big(\h,\{\frkR_{0,q}\}_{q=1}^{+\infty}\big)\stackrel{\iota}{\longrightarrow}\big(\g\oplus \h,\{\huaR_k\}_{k=1}^{+\infty}\big)\stackrel{p}\longrightarrow\big(\g,\{l_k\}_{k=1}^{+\infty}\big)\longrightarrow0.$$
		Here the linear maps $\iota$ and $p$ are given by
		\begin{eqnarray}\label{open-closed-to-extension}
			\iota(u)=(0,u),\,\,\,p(x,u)=x,\,\,\,\forall x\in\g,\,\,\,u\in\h.
		\end{eqnarray}
		Since $\bar{\Sym}(\h)$ is an $L_\infty$-subalgebra of $\big(\bar{\Sym}(\g\oplus\h),\bar{\Delta},e^{[\cdot,\hat{\Theta}]_C}\Phi\big)$, it is clear that the embedding map $$i:\big(\bar{\Sym}(\h),\bar{\Delta},e^{[\cdot,\hat{\Theta}]_C}\Phi|_{\bar{\Sym}(\h)}\big)\lon \big(\bar{\Sym}(\g\oplus\h),\bar{\Delta},e^{[\cdot,\hat{\Theta}]_C}\Phi\big)$$
		is an $L_\infty$-algebra homomorphism. By  \eqref{eq:relation},   $\bar{p}\circ e^{\hat{\Theta}}\circ i$ is an $L_\infty$-algebra homomorphism from $(\bar{\Sym}(\h),\bar{\Delta},e^{[\cdot,\hat{\Theta}]_C}\Phi|_{\bar{\Sym}(\h)})$ to $(\bar{\Sym}(\g),\bar{\Delta},\Psi(\sum_{k=1}^{+\infty}l_k))$. For all $u_1,\ldots,u_n\in\h,$ we have
		\begin{eqnarray*}
			\pr_{\g}\big((\bar{p}\circ e^{\hat{\Theta}}\circ i)(u_1 \ldots  u_n)\big)&=&\pr_{\g}\big(\bar{p}( e^{\hat{\Theta}}(u_1 \ldots u_n))\big)
			\stackrel{\eqref{iso-coalgebra},\eqref{open-closed-to-extension}}{=}\Theta_n(u_1 \ldots u_n),
		\end{eqnarray*}
		which implies that $\bar{\Theta}=\bar{p}\circ e^{\hat{\Theta}}\circ i$. Thus, we deduce that $\Theta$ is an $L_\infty$-algebra homomorphism from $(\h,\{\alpha_k\}_{k=1}^{+\infty})$ to $(\g,\{l_k\}_{k=1}^{+\infty})$. The proof is finished.
	\end{proof}

	\begin{rmk}
		Theorem \ref{twist-homotopy-lie} above generalizes   \cite[Proposition 3.10]{CC} and    \cite[Proposition 5.11]{LST}.
	\end{rmk}

	We now give an explicit description of homotopy Rota-Baxter operators.
	
	\begin{pro}\label{twist-homotopy-lie-homomorphism}
		$\Theta=\sum_{k=1}^{+\infty}\Theta_k\in \Hom^0(\bar{\Sym}(\h),\g)$ is a homotopy    Rota-Baxter operator on the open-closed homotopy Lie algebra $(\g,\h,\{l_k\}_{k=1}^{+\infty},\{\frkR_{p,q}\}_{p+q\ge1})$ if and only if the following equalities hold for all $n\geq 1$ and all homogeneous elements $u_1,\ldots,u_n\in \h$,
{\footnotesize
		\begin{eqnarray}\label{homotopy-rota-baxter-equivalence}
			\nonumber&&\sum_{k=1}^{n}\sum_{i_1+\ldots+i_k=n\atop i_1,\ldots,i_k\ge1}\sum_{\sigma\in\mathbb S_{(i_1,\ldots,i_k)}}\frac{\varepsilon(\sigma)}{k!}l_{k}\Big(\Theta_{i_1}(u_{\sigma(1)} \ldots  u_{\sigma(i_1)}),\ldots, \Theta_{i_k}(u_{\sigma(i_1+\ldots+i_{k-1}+1)} \ldots  u_{\sigma(n)})\Big)\\
			&&=\sum_{k=1}^{n}\sum_{\sigma\in\mathbb S_{(k,n-k)}}\varepsilon(\sigma)\Theta_{n-k+1}\Big(\frkR_{0,k}(u_{\sigma(1)} \ldots  u_{\sigma(k)}) u_{\sigma(k+1)} \ldots u_{\sigma(n)}\Big)\\
			\nonumber&&+\sum_{k=1}^{n}\sum_{i_1+\ldots+i_{k+1}=n\atop i_1,\ldots,i_k\ge1,i_{k+1}\ge0}\sum_{\sigma\in\mathbb S_{(i_1,\ldots,i_{k+1})}}\frac{\varepsilon(\sigma)}{k!}\\
			\nonumber &&\Theta_{i_{k+1}+1}\Big(\frkR_{k,0}\big(\Theta_{i_1}(u_{\sigma(1)} \ldots  u_{\sigma(i_1)}) \ldots \Theta_{i_{k}}(u_{\sigma(i_1+\ldots+i_{k-1}+1)} \ldots  u_{\sigma(i_1+\ldots+i_{k})})\big) u_{\sigma(i_1+\ldots+i_{k}+1)} \ldots u_{\sigma(n)}\Big)\\
			\nonumber&&+\sum_{k=1}^{n-1}\sum_{i_1+\ldots+i_{k+2}=n\atop i_1,\ldots,i_{k+1}\ge 1,i_{k+2}\ge0}\sum_{\sigma\in\mathbb S_{(i_1,\ldots,i_{k+2})}}\frac{\varepsilon(\sigma)}{k!}\\
			\nonumber &&\Theta_{i_{k+2}+1}\Big(\frkR_{k,i_{k+1}}\big(\Theta_{i_1}(u_{\sigma(1)} \ldots u_{\sigma(i_1)}) \ldots \Theta_{i_{k}}(u_{\sigma(i_1+\ldots+i_{k-1}+1)} \ldots u_{\sigma(i_1+\ldots+i_{k})})\otimes u_{\sigma(i_1+\ldots+i_{k}+1)} \ldots u_{\sigma(i_1+\ldots+i_{k+1})}\big)\\
			\nonumber&&u_{\sigma(i_1+\ldots+i_{k+1}+1)} \ldots u_{\sigma(n)}\Big).
		\end{eqnarray}
}
	\end{pro}
	
	\begin{proof}
		A degree 0 element  $\Theta=\sum_{k=1}^{+\infty}\Theta_k\in \Hom^0(\bar{\Sym}(\h),\g)$  is a homotopy    Rota-Baxter operator on the open-closed homotopy Lie algebra $(\g,\h,\{l_k\}_{k=1}^{+\infty},\{\frkR_{p,q}\}_{p+q=1}^{+\infty})$ if and only if for all $n\geq 1$  and  all homogeneous elements $u_1,\ldots,u_n\in \h$, we have
		\begin{eqnarray*}
			\pr_{\g}\big(e^{[\cdot,\hat{\Theta}]_C}\Phi\big)(u_1 \ldots  u_n)=0.
		\end{eqnarray*}
By \eqref{twistor},~\eqref{iso-coalgebra} and   similar calculations as the proof of Theorem \ref{twist-homotopy-lie}, one can show that this is  equivalent to \eqref{homotopy-rota-baxter-equivalence}.
	\end{proof}

	\subsection{Post-Lie$_\infty$ algebras and homotopy Rota-Baxter operators}
	In this subsection, we study the relation between post-Lie$_\infty$ algebras and  homotopy Rota-Baxter operators on  the open-closed homotopy Lie algebras given by $L_\infty$-actions of $L_\infty$-algebras.
	\emptycomment{Homotopy relative Rota-Baxter operators (also called $\huaO$-operators) on $L_\infty$-algebras with respect to actions of $L_\infty$-algebras were first studied by Caseiro and Nunes da Costa in \cite{CC},  which are homotopy   Rota-Baxter operators  on  the open-closed homotopy Lie algebras given by actions of $L_\infty$-algebras.}
	
	Observe that a post-Lie$_\infty$ algebra naturally gives rise to a homotopy Rota-Baxter operator.

	\begin{thm}\label{homotopy-pre-lie-id-o}
		Let $(\g,\{\frkM_{p,q}\}_{p\ge0,q\ge1})$ be a post-Lie$_\infty$ algebra. Then the identity map $\Id:\g\lon\g$ is a   homotopy Rota-Baxter operator on the open-closed homotopy Lie algebra  $(\g,\g,\{l^C_k\}_{k=1}^{+\infty},\{\frkM_{p,q}\}_{p+q\ge1})$.
	\end{thm}
	
	\begin{proof}
		By Theorem \ref{homotopy-post-lie-to-lie-cor},  $(\g,\g,\{l^C_k\}_{k=1}^{+\infty},\{\frkM_{p,q}\}_{p+q\ge1})$ is an open-closed homotopy Lie algebra. By \eqref{sub-homotopy-lie}, we deduce that $\Id$ is a  homotopy Rota-Baxter operator on   the open-closed homotopy Lie algebra  $(\g,\g,\{l^C_k\}_{k=1}^{+\infty},\{\frkM_{p,q}\}_{p+q=1}^{+\infty})$.
	\end{proof}

	Now we show that homotopy Rota-Baxter operators on the open-closed homotopy Lie algebras given by $L_\infty$-actions of $L_\infty$-algebras  induce post-Lie$_\infty$ algebras. This generalizes \cite[Proposition 5.1]{Aguiar}, \cite[Theorem 5.4]{BGN} and  \cite[Theorem 3.11]{TBGS}.
	
	\begin{thm}\label{homotopy-RB-homotopy-post-lie}
		Let $\Theta=\sum_{k=1}^{+\infty}\Theta_k\in \Hom(\bar{\Sym}(\h),\g)$ be a homotopy   Rota-Baxter operator on an open-closed homotopy Lie algebra given by an $L_\infty$-action $\rho=\{\rho_k\}_{k=1}^{+\infty}$ of an $L_\infty$-algebra $(\g,\{l_k\}_{k=1}^{+\infty})$ on an $L_\infty$-algebra $(\h,\{\mu_k\}_{k=1}^{+\infty})$. Then $(\h,\{\frkM_{p,q}\}_{p\ge0,q\ge1})$ is a post-Lie$_\infty$ algebra, where $\frkM_{0,q}=\mu_q$ and
		\begin{eqnarray*}
			&&\frkM_{p,q}(u_1\ldots u_p\otimes u_{p+1} \ldots u_{p+q})\\
			&&=\sum_{k=1}^{p}\sum_{i_1+\ldots+i_{k}=p\atop i_1,\ldots,i_{k}\ge 1}\sum_{\sigma\in\mathbb S_{(i_1,\ldots,i_{k})}}\frac{\varepsilon(\sigma)}{k!}\frkR_{k,q}\Big(\Theta_{i_1}(u_{\sigma(1)} \ldots  u_{\sigma(i_1)}) \ldots \Theta_{i_{k}}(u_{\sigma(i_1+\ldots+i_{k-1}+1)} \ldots  u_{\sigma(p)})\otimes u_{p+1} \ldots u_{p+q}\Big),
		\end{eqnarray*}
		where the maps $\{\frkR_{p,q}\}_{p\ge1,q\ge1}$ are given by \eqref{open-closed-to-lie-homo}.
	\end{thm}
	\begin{proof}
		First we observe that for   $\{\frkM_{p,q}\}_{p\ge0,q\ge1}$ defined above,   $\{l^C_n\}_{n=1}^{+\infty}$ defined by \eqref{sub-homotopy-lie} is exactly the $L_\infty$-algebra structure $(\h,\{\alpha_n\}_{n=1}^{+\infty})$   given in Theorem \ref{twist-homotopy-lie}.
		
		By Theorem \ref{twist-homotopy-lie}, $\bar{\Theta}$ is a homomorphism from the dg coalgebra $(\bar{\Sym}(\h),\bar{\Delta},e^{[\cdot,\hat{\Theta}]_C}\Phi|_{\bar{\Sym}(\h)})$ to the dg coalgebra $(\bar{\Sym}(\g),\bar{\Delta},\Psi(\sum_{k=1}^{+\infty}l_k))$. Moreover, since $\rho$ is an  $L_\infty$-action of the $L_\infty$-algebra $(\g,\{l_k\}_{k=1}^{+\infty})$ on the  $L_\infty$-algebra $(\h,\{\mu_k\}_{k=1}^{+\infty})$, we deduce that $\bar{\rho}\circ \bar{\Theta}:\bar{\Sym}(\h)\lon \bar{\Sym}(s^{-1}\overline{\coDer}(\Sym(\h)))$ is a homomorphism of dg coalgebras. Thus,
		$\bar{\rho}\circ \bar{\Theta}$ is an $L_\infty$-action of $(\h,\{\alpha_k\}_{k=1}^{+\infty})$ on $(\h,\{\mu_k\}_{k=1}^{+\infty})$. Equivalently, $(\h,\h,\{l^C_n\}_{n=1}^{+\infty},\{\frkM_{p,q}\}_{p+q\ge1})$ is an open-closed homotopy Lie algebra such that  $\frkM_{p,0}=0$ for all $p\ge1$. By Theorem \ref{th:characterization-more}, we deduce that $\{\frkM_{p,q}\}_{p\ge0,q\ge1}$ is a post-Lie$_\infty$ algebra structure on $\h$.
	\end{proof}

	\vspace{2mm}
	\noindent
	{\bf Acknowledgements.} The first author (A. Lazarev) was partially supported by the EPSRC grant EP/T029455/1. The second and third authors (Y. Sheng and R. Tang) were partially supported by NSFC (12471060, W2412041, 12371029) and the Fundamental Research Funds for the Central Universities.
	This work was completed in part while R. Tang was visiting the Key Laboratory of Mathematics and Its Applications of Peking University and he wishes to thank this laboratory for excellent working conditions.  R. Tang also thanks Xiaomeng Xu for the hospitality  during his stay at Peking University.

\end{document}